\documentclass[10pt]{article}
\usepackage{lmodern}
\usepackage{amsmath}
\usepackage[T1]{fontenc}
\usepackage[utf8]{inputenc}
\usepackage{authblk}
\usepackage{amsfonts}
\usepackage{graphicx}
\usepackage{color,tikz}

\usepackage{rotating}
\usepackage{amssymb}
\usepackage[english]{babel}
\usepackage{color}
\usepackage{amsthm}
\usepackage{graphicx}
\usepackage{mathrsfs}
\usepackage{makecell}
\usepackage{microtype}
\usepackage{enumerate}
\usepackage{mathscinet}
\usepackage{array}
\usepackage{multirow}
\usepackage{booktabs}
\usepackage{enumerate}
\usepackage[cal=boondoxo,bb=ams]{mathalfa}
\usepackage{hyperref}
\hypersetup{hidelinks}
\usepackage{titlesec}
\newtheorem{defn}{Definition}[section]
\newtheorem{hyp}{Hypothesis}
\newtheorem{theorem}{Theorem}[section]
\newtheorem{prop}{Proposition}[section]
\newtheorem{lemma}{Lemma}[section]
\newtheorem{coro}{Corollary}[section]
\newtheorem{remark}{Remark}[section]
\newtheorem{exam}{Example}[section]

\newcommand{\ml}{\mathcal}
\newcommand{\mb}{\mathbb}

\DeclareMathOperator{\intt}{int}
\DeclareMathOperator{\extt}{ext}
\DeclareMathOperator{\bdd}{bdd}

\title{Decay properties and asymptotic behaviors for a wave equation with general strong damping}
\author{Wenhui Chen\thanks{Wenhui Chen (wenhui.chen.math@gmail.com)}}
\author[2]{Ryo Ikehata\thanks{Ryo Ikehata (ikehatar@hiroshima-u.ac.jp)}}
\affil[1]{School of Mathematical Sciences, Shanghai Jiao Tong University, 200240 Shanghai, China}
\affil[2]{Department of Mathematics, Division of Educational Sciences, Graduate School of Humanities and Social Sciences, Hiroshima University, 739-8524 Higashi-Hiroshima, Japan}

\date{}
\begin{document}

\maketitle
\begin{abstract}
	\medskip
In this paper, we study the Cauchy problem for a wave equation with general strong damping $-\mu(|D|)\Delta u_t$ motivated by [Tao, \emph{Anal. PDE} (2009)] and [Ebert-Girardi-Reissig, \emph{Math. Ann.} (2020)]. By employing energy methods in the Fourier space and WKB analysis, we derive decay estimates for solutions under a large class of $\mu(|D|)$. In particularly, a threshold $\lim\nolimits_{|\xi|\to\infty}\mu(|\xi|)=\infty$ is discovered for the regularity-loss phenomenon, where $\mu(|\xi|)$ denotes the symbol of $\mu(|D|)$. Furthermore, we investigate different asymptotic profiles of solution with additionally $L^1$ initial data, where some refined estimates in the sense of enhanced decay rate and reduced regularity are found. The derived results almost cover the known results with sufficiently small loss.\\
	
	\noindent\textbf{Keywords:} wave equation, general strong damping, Cauchy problem, decay properties, regularity-loss phenomenon, asymptotic profiles.\\
	
	\noindent\textbf{AMS Classification (2020)} 35L05, 35L15, 35B40
\end{abstract}
\fontsize{12}{15}
\selectfont
\section{Introduction}
\subsection{Background of strongly damped wave models}
In recent twenty years, the strongly damped wave equation (or the so-called viscoelastic damped wave equation), namely,
\begin{align}\label{Strong_Damped_Waves}
	u_{tt}-\Delta u-\Delta u_t=0,
\end{align}
catches a lot of attentions. Indeed, the model \eqref{Strong_Damped_Waves} originates from a linearized Kuznetsov's equation, which is a fundamental model in nonlinear acoustics for describing high-intensity ultrasonic waves. Kuznetsov's equation was established by Lighthill's scheme approximation procedures for Navier-Stokes-Fourier system under a given irrotational flow. It has some applications in medical imaging and therapy, ultrasound cleaning and welding (see, for example, \cite{Abramov=1999}). Because of the higher-order derivative $-\Delta u_t$, the model \eqref{Strong_Damped_Waves} does not belong to hyperbolic equations anymore, but parabolic effect arises. This vital term will bring some difficulties in mathematical analysis. We refer interested reads to \cite[Subection 11.4.5, Section 14.3, Subsection 19.5.1]{Ebert-Reissig-book} for some basic knowledge for \eqref{Strong_Damped_Waves}.

At the first stage of the corresponding Cauchy problem 
\begin{align}\label{Cauchy_Strong_Damped}
\begin{cases}
u_{tt}-\Delta u-\Delta u_t=0,&x\in\mb{R}^n,\ t>0,\\
u(0,x)=u_0(x),\ u_t(0,x)=u_1(x),&x\in\mb{R}^n,
\end{cases}
\end{align}
 the pioneering works \cite{Ponce=1985,Shibata=2000} investigated some $L^p-L^q$ decay properties of solutions with suitable $1\leqslant p\leqslant q\leqslant\infty$. Moreover, the authors of \cite{D'Abbicco-Reissig=2014} obtained some decay estimates of solution with $L^2\cap L^1$ data. The additional $L^1$ regularity for initial data is the bringer of additional decay rate. Later, some asymptotic profiles of solution in an abstract framework have been derived by \cite{Ikehata-Todorova-Yordanov=2013} with the aid of the spectral analysis and energy estimates. To be specific, the asymptotic profile is depicted by the diffusion-wave solution, i.e.
 \begin{align*}
 u(t,x)\sim u_{\mathrm{diff}}(t,x):=\mathrm{e}^{-\frac{\ml{A}t}{2}}\left(\cos(\ml{A}^{1/2}t)u_0(x)+\frac{\sin(\ml{A}^{1/2}t)}{\ml{A}^{1/2}}u_1(x)\right),
 \end{align*}
where the linear operator can be opted by $\ml{A}=-\Delta$ in the consideration of the Cauchy problem \eqref{Cauchy_Strong_Damped}. Concerning optimal $L^2$-estimates in the framework of weighted $L^1$ initial data, \cite{Ikehata=2014} and \cite{Ikehata-Onodera=2017} obtained the results for higher-dimensions ($n\geqslant 3$) and lower-dimensions $(n=1,2)$, respectively, by using the refined Fourier analysis. A further topic on higher-order profiles for \eqref{Cauchy_Strong_Damped} has been done by \cite{Barrera-Volkmer=2019,Barrera-Volkmer=2021,Michihisa=2021}. So far the qualitative properties of solutions to the initial value problem \eqref{Cauchy_Strong_Damped} are well-understood.

Recently, strongly motivated by the pioneering paper \cite{Ghisi-Gobbino-Haraux=2016}, the following Cauchy problem:
\begin{align}\label{Cauchy_Very_Strong_Damped}
	\begin{cases}
		u_{tt}-\Delta u+(-\Delta)^{\theta} u_t=0,&x\in\mb{R}^n,\ t>0,\\
		u(0,x)=u_0(x),\ u_t(0,x)=u_1(x),&x\in\mb{R}^n,
	\end{cases}
\end{align}
with $\theta\in(1,\infty)$ has been taken into considerations by \cite{Ikehata-Iyota=2018,Fukushima-Ikehata-Michihisa=2021,Charao-Torres-Ikehata=2020} and references therein. They found regularity-loss decay properties of energy for any $\theta\in(1,\infty)$. The regularity-loss structure, which means that to obtain decay estimates we require higher-regularity of initial data, was primarily discovered by Professor Shuichi Kawashima and his coauthors in the studies of the dissipative Timoshiko system \cite{Ide-Haramoto-Kawashima=2008} and hyperbolic-elliptic systems \cite{Hosono-Kawashima=2006}.

\subsection{Motivation of general strong damping}
Summarizing the previous results, the authors in the literature successfully obtained a threshold $\theta_{\mathrm{reg}}=1$ in the scale $\{(-\Delta)^{\theta}\}_{\theta\geqslant0}$ (or the equivalent form $\{|D|^{2\theta}\}_{\theta\geqslant0}$) of the damping term $(-\Delta)^{\theta}u_t$. In other words, when $\theta\in[0,\theta_{\mathrm{reg}}]$ in the Cauchy problem \eqref{Cauchy_Very_Strong_Damped}, one may derive some stabilities without any regularity-loss (especially, an exponential decay for large frequencies).  For another, decay properties of regularity-loss occurs if one considers $\theta\in(\theta_{\mathrm{reg}},\infty)$.  

Nevertheless, the scale $\{|D|^{2\theta}\}_{\theta\geqslant0}$ is too rough to verify the critical damping in the sense of regularity-loss structure. Motivated by the innovative works \cite{Tao=2009} for logarithmic type operator in Navier-Stokes equations and \cite{Ebert-Girardi-Reissig=2020} for general scale of nonlinear terms in the classical damped wave model, we would like to describe the threshold of regularity-loss structure by the symbol of $\mu(|D|)$ in the damping term $-\mu(|D|) \Delta u_t$. That is to say that we are interested in the model 
\begin{align}\label{Eq_General_Damped_Waves}
	u_{tt}-\Delta u-\mu(|D|) \Delta u_t=0
\end{align}
under some assumptions for $\mu(|D|)$.

Another motivation for the investigation of \eqref{Eq_General_Damped_Waves} is to seize some effects from the general damping term, which may cause new phenomena. To the best of authors' knowledge, not only decay properties but also asymptotic profiles of solutions are completely unknown. We will partly answer these interesting questions in the present paper by taking some hypotheses for the pseudo-differential operator $\mu(|D|)$.

\subsection{Main purposes of the paper}
 Motivated by the recent studies on strongly damped wave models, we consider the following wave equation with general strong damping in the present work:
\begin{align}\label{Eq_Linear_General_Damped_Waves}
\begin{cases}
u_{tt}-\Delta u-\mu(|D|)\Delta u_t=0,&x\in\mb{R}^n,\ t>0,\\
u(0,x)=u_0(x),\ u_t(0,x)=u_1(x),&x\in\mb{R}^n,
\end{cases}
\end{align}
where the differential operator $\mu(|D|)$ of strong damping with the pseudo-differential operator $|D|$ (carrying its symbol $|\xi|$) can be understood by the mean of Fourier transform
\begin{align*}
	\ml{F}\left( \mu(|D|)f(x)\right):=\mu(|\xi|)\hat{f}(\xi)
\end{align*} 
for $f\in\ml{Z}'$, with a non-negative and continuous (with respect to $|\xi|$) function $\mu(|\xi|)$. In the last statement, $\ml{Z}'$ stands for the topological dual space to the subspace of the Schwartz space $\ml{S}$ consisting of function $\mathrm{d}_{\xi}^k\hat{f}(0)=0$ for all $k\in\mb{N}_0$, in other words, $\ml{Z}'$ is the factor space $\ml{S}'/\ml{P}$, in which $\ml{P}$ is the space of all polynomials. Our main interest in the paper is to investigate some qualitative properties of solutions to the Cauchy problem \eqref{Eq_Linear_General_Damped_Waves} perturbed by the differential operator $\mu(|D|)$.

As we mentioned in the preceding part of the text, the general pseudo-differential operator $\mu(|D|)$ is treated as a perturbation part of strong damping $-\Delta u_t$ so that it is reasonable to assume:
\begin{hyp}\label{Hyp_Diss}
	Let us suppose that $\mu=\mu(r)$ is a non-negative function such that $\mu\in\ml{C}([0,\infty))$. Moreover, it satisfies
	\begin{align}\label{A_Hyp}
		\lim\limits_{r\downarrow0}r\mu(r)=0\ \ \mbox{as well as}\ \ \lim\limits_{r\to\infty}\frac{1}{r\mu(r)}=0.
	\end{align}
\end{hyp}
Frankly speaking, Hypothesis \ref{Hyp_Diss} allows us to consider a large class of differential operator $\mu(|D|)$ in the Cauchy problem \eqref{Eq_Linear_General_Damped_Waves}. Let us give some examples.
\begin{exam}
	Hypothesis \ref{Hyp_Diss} holds for the following function $\mu=\mu(r)$:
	\begin{enumerate}[(1)]
		\item fractional type: $\mu(r)=pr^{\pm\epsilon}$ with $\epsilon\in[0,1)$ and $p\in(0,\infty)$;
		\item oscillating type: $\mu(r)=p(1+\sin r)+q(1+\cos r)$ with $p,q\in(0,\infty)$;
		\item logarithmic type: $\mu(r)=(\log(1+r))^{\gamma}$ with $\gamma\in(-1,\infty)$;
		\item k-logarithmic type: $\mu(r)=\underbrace{\log\big(1+\log\big(1+\cdots\log\big(\log}_{k\ times \ \log}(1+r)\big)\big)\big)$ with $k\in\mb{N}$.
	\end{enumerate}
	Moreover, we also can consider a function $\mu\in\ml{C}([0,\infty))$ but $\mu\not\in\ml{C}^1([0,\infty))$, for instance,
	\begin{align*}
		\mu(r)=\begin{cases}
			(r-1)^2\sin(\frac{1}{r-1})&\mbox{if}\ \ r\neq1,\\
			0&\mbox{if}\ \ r=1.
		\end{cases}
	\end{align*}
\end{exam}
\begin{remark}
The second assumption of \eqref{A_Hyp} in Hypothesis \ref{Hyp_Diss} may be changed into $\lim\nolimits_{r\to\infty}r\mu(r)\geqslant c\geqslant 0$. Then, we can derive an exponential stability for large frequencies. More specific explanations will be shown in the end of Subsection \ref{Section_Decay_Estimates}. A typical example, which can be no contained in Hypothesis \ref{Hyp_Diss}, takes the form $\mu(r)=r^{-2}\log(1+r^{2\sigma})$ with $\sigma\in(1/2,\infty)$ (cf. \cite{Charao-DAbbicco-Ikehata}), but it adapts Hypothesis \ref{Hyp_Diss_C}.
\end{remark}

Our first purpose in the present paper is to understand decay properties of solutions to the damped wave equation \eqref{Eq_Linear_General_Damped_Waves} under Hypothesis \ref{Hyp_Diss}. In Section \ref{Section_Decay_Prop}, by developing two crucial lemmas to describe decay rates of some Fourier multipliers, we derive decay estimates of solutions with the aid of energy methods in the Fourier space. Among these results, a new threshold related to the symbol of $\mu(|D|)$, i.e.
\begin{align}\label{Cond_Regularity_Introdu}
\lim\limits_{|\xi|\to\infty}\mu(|\xi|)=\infty,
\end{align}
for determining regularity-loss decay properties is discovered. Namely, the regularity-loss phenomenon will vanish away \emph{if and only if} the condition \eqref{Cond_Regularity_Introdu} does not hold.

Furthermore, to derive sharper estimates for the solution itself in lower-dimensions and asymptotic profiles for large-time, we apply asymptotic expansions associated with WKB analysis. In Section \ref{Section_Asym_Prof}, we demonstrate the asymptotic profiles of solution in the $L^2$ norm for $t\gg1$ as follows:
\begin{align*}
u(t,x)&\sim u_{\mathrm{loss}}(t,x):= \chi_{\intt}(D)\mathrm{e}^{-\frac{1}{2}\mu(|D|)|D|^2t}\left(\cos(|D|t)u_0(x)+\frac{\sin(|D|t)}{|D|}u_1(x)\right)\\
&\qquad\qquad\quad\ \  \quad+(1-\chi_{\intt}(D))\mathrm{e}^{-\frac{t}{\mu(|D|)}}\left(u_0(x)+\frac{1}{\mu(|D|)|D|^2}u_1(x)\right)\ \ \mbox{if}\ \ \lim\limits_{r\to\infty}\mu(r)=\infty,
\end{align*}
and
\begin{align*}
	u(t,x)&\sim u_{\mathrm{nolo}}(t,x):=\mathrm{e}^{-\frac{1}{2}\mu(|D|)|D|^2t}\left(\cos(|D|t)u_0(x)+\frac{\sin(|D|t)}{|D|}u_1(x)\right)\ \ \mbox{if}\ \ \lim\limits_{r\to\infty}\mu(r)<\infty.
\end{align*}
For this reason, the condition \eqref{Cond_Regularity_Introdu} is also an important threshold for clarifying different asymptotic profiles. By subtracting the corresponding asymptotic profiles ($u_{\mathrm{loss}}(t,x)$ when regularity-loss; $u_{\mathrm{nolo}}(t,x)$ when regularity-no-loss), we will observe some enhanced decay rates and reduced regularities in comparison with the estimate for the solution itself. We point out these nomenclatures \emph{enhanced decay rate} and \emph{reduced regularity} throughout this paper in the sense that:
\begin{itemize}
	\item enhanced decay rate $(1+t)^{-\alpha}$ with $\alpha>0$ is valid, if $\|\ml{A}u_0\|\lesssim (1+t)^{-\gamma}\|u_0\|$ and $\|\ml{A}u_0-\ml{B}u_0\|\lesssim (1+t)^{-\gamma-\alpha}\|u_0\|$ hold;
    \item  reduced (Sobolev) regularity $\dot{H}^{-\sigma}$ with $\sigma>0$ is valid, if $\|\ml{A}u_0\|\lesssim \|u_0\|_{\dot{H}^s}$ and $\|\ml{A}u_0-\ml{B}u_0\|\lesssim \|u_0\|_{\dot{H}^{s-\sigma}}$ hold;
\end{itemize}
 where $\ml{A}$ and $\ml{B}$ stand for the objective solution operator for the equation and our constructed profile, respectively.
 
 Our main contributions in the present paper consist in deriving $L^2$ estimates of solutions (see Theorems \ref{Thm_Energy_Decay} and \ref{Thm_Solution_Itself}), as well as investigating large-time profiles (see Theorem \ref{Thm_Asymptotic_Profiles}) to the wave equation carrying general damping \eqref{Eq_General_Damped_Waves}. These results almost (in the sense of sufficiently small loss) cover the known theorems in literature \cite{Ikehata-Natsume=2012,Charao-daLuz-Ikehata=2013,Ikehata=2014,D'Abbicco-Reissig=2014,Ikehata-Onodera=2017,Ikehata-Iyota=2018,Charao-Torres-Ikehata=2020,Charao-DAbbicco-Ikehata,Fukushima-Ikehata-Michihisa=2021}. We will take them as examples displaying after the statements of our theorems.


\medskip
\noindent\textbf{Notations:} Throughout this manuscripts, $|D|^s$ and $\langle D\rangle^s$ with $s\geqslant0$ stand for the pseudo-differential operators with symbol $|\xi|^s$ and $\langle \xi\rangle^s$, respectively, carrying the Japanese bracket $\langle\xi\rangle^2:=1+|\xi|^2$. Next, let us define the following zones from the Fourier space
\begin{align*}
	\ml{Z}_{\intt}(\varepsilon)&:=\{\xi\in\mb{R}^n:|\xi|\leqslant\varepsilon\ll1\},\\
	\ml{Z}_{\bdd}(\varepsilon,N)&:=\{\xi\in\mb{R}^n:\varepsilon\leqslant |\xi|\leqslant N\},\\
	\ml{Z}_{\extt}(N)&:=\{\xi\in\mb{R}^n:|\xi|\geqslant N\gg1\}.
\end{align*}  The cut-off functions $\chi_{\intt}(\xi),\chi_{\bdd}(\xi),\chi_{\extt}(\xi)\in \mathcal{C}^{\infty}$ endowing their supports in the zone $\ml{Z}_{\intt}(\varepsilon)$, $\ml{Z}_{\bdd}(\varepsilon/2,2N)$ and $\ml{Z}_{\extt}(N)$, individually, fulfilling $\chi_{\intt}(\xi)+\chi_{\bdd}(\xi)+\chi_{\extt}(\xi)=1$ for all $\xi \in \mb{R}^n$. Furthermore, let us introduce a function space $\dot{H}^{s}_{\mu,\ell}$ related to the homogeneous Sobolev space with $s\in\mb{R}$  and $\ell\geqslant0$ by
\begin{align*}
	\dot{H}^s_{\mu,\ell}:=\left\{f\in\ml{S}'/\ml{P}:\|f\|_{\dot{H}^s_{\mu,\ell}}:=\|\mu(|D|)^{\ell}|D|^sf\|_{L^2}<\infty \right\},
\end{align*} 
where $\mu=\mu(r)\geqslant0$ is assumed to be a continuous function. Particularly, in the case $\mu(r)=r^{\beta}$, we claim $\dot{H}^{s}_{\mu,\ell}=\dot{H}^{s+\beta\ell}$. Hereafter, $c$ and $C$ denote some constants that may be changed from line to line. The symbol $f\lesssim g$ means that there exists a positive constant $C$ fulfilling $f\leqslant Cg$.
\section{Decay properties with additionally $L^1$ initial data}\label{Section_Decay_Prop}
Due to the general operator $\mu(|D|)$ in the evolution equation  of \eqref{Eq_General_Damped_Waves}, the operator
\begin{align*}
	\ml{L}:=\partial_t^2-\Delta-\mu(|D|)\Delta\partial_t 
\end{align*}
is not a hyperbolic operator, even not a $p$-evolution operator (see \cite[Chapter 3, Definition 3.2]{Ebert-Reissig-book}). We may not apply the general theory for such classes of operators. In this section, we will estimate the solution and its higher-order derivatives in the $L^2$ norm by employing suitable energy methods in the Fourier space rather than explicit solution formula. 

\subsection{Estimates by energy method in the Fourier space}
To begin, let us apply the partial Fourier transform with respect to spatial variables $x$ for the linear Cauchy problem \eqref{Eq_Linear_General_Damped_Waves} such that $\hat{u}(t,\xi):=\ml{F}_{x\to\xi}(u(t,x))$. Afterwards, we are able to obtain
\begin{align}\label{Eq_Fourier_Linear}
\begin{cases}
\hat{u}_{tt}+\mu(|\xi|)|\xi|^2\hat{u}_t+|\xi|^2\hat{u}=0,&\xi\in\mb{R}^n,\ t>0,\\
\hat{u}(0,\xi)=\hat{u}_0(\xi),\ \hat{u}_t(0,\xi)=\hat{u}_1(\xi),&\xi\in\mb{R}^n.
\end{cases}
\end{align}

We now state energy estimates for the Cauchy problem \eqref{Eq_Fourier_Linear} in the below.
\begin{prop}\label{Prop_Pointwise}
Let us assume $\mu\in\ml{C}([0,\infty))$. Then, the following pointwise estimates for the energy terms of the Cauchy problem \eqref{Eq_Fourier_Linear} hold:
\begin{align*}
|\hat{u}_t(t,\xi)|^2+|\xi|^2|\hat{u}(t,\xi)|^2\lesssim \mathrm{e}^{-c\rho(|\xi|)t}\left(|\xi|^2|\hat{u}_0(\xi)|^2+|\hat{u}_1(\xi)|^2\right),
\end{align*}
with $c>0$ for any  $\xi\in\mb{R}^n$ and $t>0$, where the key function locating in the exponent is expressed by
\begin{align}\label{Key_function}
	\rho(|\xi|):=\frac{|\xi|^2\mu(|\xi|)}{1+|\xi|^2\mu(|\xi|)^2}.
\end{align}
\end{prop}
\begin{remark}
The key function $\rho(|\xi|)$ in \eqref{Key_function}, indeed, are able to describe decay properties of solutions. Under Hypothesis \ref{Hyp_Diss}, we approximate it by
\begin{align}\label{Behavior_of_Key_Function}
\rho(|\xi|)\approx\begin{cases}
|\xi|^2\mu(|\xi|)&\mbox{for}\ \ \xi\in\ml{Z}_{\intt}(\varepsilon),\\
c>0&\mbox{for}\ \ \xi\in\ml{Z}_{\bdd}(\varepsilon,N),\\
\mu(|\xi|)^{-1}&\mbox{for}\ \ \xi\in\ml{Z}_{\extt}(N),
\end{cases}
\end{align}
with $\varepsilon\ll 1$ and $N\gg 1$. Later, we will use the aforementioned behavior of $\rho(|\xi|)$ in each frequency zone to derive decay estimates of solutions in the $L^2$ norm.
\end{remark}
\begin{proof}[Proof of Proposition \ref{Prop_Pointwise}]
First of all, let us introduce an energy functional
\begin{align*}
	E_0[\hat{u}](t,\xi):=|\hat{u}_t(t,\xi)|^2+|\xi|^2|\hat{u}(t,\xi)|^2,
\end{align*}
and some auxiliary functionals
\begin{align*}
	E[\hat{u}](t,\xi)&:=E_0[\hat{u}](t,\xi)+2\beta\rho(|\xi|)\Re\left(\hat{u}(t,\xi)\bar{\hat{u}}_t(t,\xi)\right)+\beta\rho(|\xi|)\mu(|\xi|)|\xi|^2|\hat{u}(t,\xi)|^2,\\
	F[\hat{u}](t,\xi)&:=\mu(|\xi|)|\xi|^2|\hat{u}_t(t,\xi)|^2+\beta\rho(|\xi|)|\xi|^2|\hat{u}(t,\xi)|^2,\\
	R[\hat{u}](t,\xi)&:=\beta\rho(|\xi|)|\hat{u}_t(t,\xi)|^2,
\end{align*}
with a suitable constant $\beta>0$ to be determined later. Then, multiplying \eqref{Eq_Fourier_Linear}$_1$ by $\bar{\hat{u}}_t$ and $\beta\rho(|\xi|)\bar{\hat{u}}$, respectively, one owns
\begin{align*}
	\frac{\mathrm{d}}{\mathrm{d}t}E_0[\hat{u}](t,\xi)+2\mu(|\xi|)|\xi|^2|\hat{u}_t(t,\xi)|^2=0,
\end{align*}
as well as
\begin{align*}
	\frac{\mathrm{d}}{\mathrm{d}t}\left(2\beta\rho(|\xi|)\Re\left(\hat{u}(t,\xi)\bar{\hat{u}}_t(t,\xi)\right)+\beta\rho(|\xi|)\mu(|\xi|)|\xi|^2|\hat{u}(t,\xi)|^2\right)+2\beta\rho(|\xi|)|\xi|^2|\hat{u}(t,\xi)|^2=2R[\hat{u}](t,\xi).
\end{align*}
The sum of last two equations shows
\begin{align*}
	\frac{\mathrm{d}}{\mathrm{d}t}E[\hat{u}](t,\xi)+2F[\hat{u}](t,\xi)=2R[\hat{u}](t,\xi).
\end{align*}

We now choose the parameter $\beta$ belonging to $(0,1)$, which follows
\begin{align}\label{Ineq_03}
	\beta\rho(|\xi|)\leqslant \beta\mu(|\xi|)|\xi|^2
\end{align}
from the setting of the key function \eqref{Key_function}. Due to the fact that
\begin{align*}
R[\hat{u}](t,\xi)&\leqslant \beta\mu(|\xi|)|\xi|^2|\hat{u}_t(t,\xi)|^2\leqslant \beta F[\hat{u}](t,\xi),
\end{align*}
we immediately obtain
\begin{align}\label{Ineq_01}
	\frac{\mathrm{d}}{\mathrm{d}t}E[\hat{u}](t,\xi)+2(1-\beta)F[\hat{u}](t,\xi)\leqslant 0.
\end{align}
Using Cauchy's inequality in the next form:
\begin{align*}
	2\Re\left(\hat{u}(t,\xi)\bar{\hat{u}}_t(t,\xi)\right)\leqslant\frac{|\hat{u}_t(t,\xi)|^2}{|\xi|}+|\xi|\,|\hat{u}(t,\xi)|^2,
\end{align*}
one may arrive at
\begin{align*}
	E[\hat{u}](t,\xi)&\leqslant E_0[\hat{u}](t,\xi)+\beta\rho(|\xi|)\left(\frac{|\hat{u}_t(t,\xi)|^2}{|\xi|}+|\xi|\,|\hat{u}(t,\xi)|^2\right)+\beta\rho(|\xi|)\mu(|\xi|)|\xi|^2|\hat{u}(t,\xi)|^2\\
	&= \left(1+\frac{\beta\rho(|\xi|)}{|\xi|}\right)|\hat{u}_t(t,\xi)|^2+\beta|\xi|^2\left(\frac{1}{\beta}+\frac{\rho(|\xi|)}{|\xi|}+\rho(|\xi|)\mu(|\xi|)\right)|\hat{u}(t,\xi)|^2.
\end{align*}
Actually, there exist (large) positive constants $M_1$ and $M_2$ such that
\begin{align}\label{Ineq_04}
	(\beta\rho(|\xi|)+|\xi|)\rho(|\xi|)&\leqslant M_1|\xi|^3\mu(|\xi|),\\
	\frac{1}{\beta}+\frac{\rho(|\xi|)}{|\xi|}+\rho(|\xi|)\mu(|\xi|)&\leqslant M_2.\label{Ineq_05}
\end{align}
\begin{remark}
In order to find such constant $M_1$, we know
\begin{align*}
(\beta\rho(|\xi|)+|\xi|)\rho(|\xi|)&=\frac{\beta|\xi|\mu(|\xi|)+1+|\xi|^2\mu(|\xi|)^2}{(1+|\xi|^2\mu(|\xi|)^2)^2}|\xi|^3\mu(|\xi|)\\
&\leqslant \frac{1+2|\xi|^2\mu(|\xi|)^2+\frac{1}{4}\beta^2}{1+2|\xi|^2\mu(|\xi|)^2+|\xi|^4\mu(|\xi|)^4}|\xi|^3\mu(|\xi|)\leqslant M_1|\xi|^3\mu(|\xi|),
\end{align*}
in which we fixed $M_1=1+\beta^2/4$. To continue, for the existence of $M_2$ we observe that
\begin{align*}
	\frac{1}{\beta}+\frac{\rho(|\xi|)}{|\xi|}+\rho(|\xi|)\mu(|\xi|)&\leqslant \frac{1}{\beta}+\frac{2|\xi|^2\mu(|\xi|)^2+\frac{1}{4}}{1+|\xi|^2\mu(|\xi|)^2}\leqslant M_2,
\end{align*}
where we chose $M_2=4+2/\beta$.
\end{remark}
For these reasons, we have
\begin{align}\label{Ineq_02}
\rho(|\xi|)E[\hat{u}](t,\xi)&\leqslant M_1|\xi|^2\mu(|\xi|)|\hat{u}_t(t,\xi)|^2+M_2\beta\rho(|\xi|)|\xi|^2|\hat{u}(t,\xi)|^2\notag\\
&\leqslant (M_1+M_2)F[\hat{u}](t,\xi).
\end{align}
The consideration of \eqref{Ineq_01} associated with \eqref{Ineq_02} leads to
\begin{align*}
	\frac{\mathrm{d}}{\mathrm{d}t}E[\hat{u}](t,\xi)+\frac{2(1-\beta)}{M_1+M_2}\rho(|\xi|)E[\hat{u}](t,\xi)\leqslant 0,
\end{align*}
which yields from Gr\"onwall's inequality that
\begin{align*}
	E[\hat{u}](t,\xi)\leqslant \mathrm{e}^{-c\rho(|\xi|)t}E[\hat{u}](0,\xi),
\end{align*}
with $c=2(1-\beta)/(M_1+M_2)>0$.

The ultimate procedure of the proof is to control the desired energy $E_0[\hat{u}](t,\xi)$ by $E[\hat{u}](t,\xi)$. For one thing, according to
\begin{align*}
	2\beta\rho(|\xi|)\Re\left(\hat{u}(t,\xi)\bar{\hat{u}}_t(t,\xi)\right)\leqslant\beta\rho(|\xi|)\mu(|\xi|)|\xi|^2|\hat{u}(t,\xi)|^2+\frac{\beta\rho(|\xi|)}{\mu(|\xi|)|\xi|^2}|\hat{u}_t(t,\xi)|^2,
\end{align*} 
and applying \eqref{Ineq_03}, one can get
\begin{align*}
	E[\hat{u}](t,\xi)&\geqslant \left(1-\frac{\beta\rho(|\xi|)}{\mu(|\xi|)|\xi|^2}\right)|\hat{u}_t(t,\xi)|^2+|\xi|^2|\hat{u}(t,\xi)|^2\\
	&\geqslant (1-\beta)E_0[\hat{u}](t,\xi).
\end{align*}
For another, because of
\begin{align*}
	2\Re\left(\hat{u}(t,\xi)\bar{\hat{u}}_t(t,\xi)\right)\leqslant \mu(|\xi|)|\xi|^2|\hat{u}(t,\xi)|^2+\frac{1}{\mu(|\xi|)|\xi|^2}|\hat{u}_t(t,\xi)|^2,
\end{align*}
it is clear that
\begin{align*}
	E[\hat{u}](t,\xi)&\leqslant E_0[\hat{u}](t,\xi)+2\beta\rho(|\xi|)\mu(|\xi|)|\xi|^2|\hat{u}(t,\xi)|^2+\frac{\beta\rho(|\xi|)}{\mu(|\xi|)|\xi|^2}|\hat{u}_t(t,\xi)|^2\\
	&\leqslant E_0[\hat{u}](t,\xi)+2\beta|\xi|^2|\hat{u}(t,\xi)|^2+\beta|\hat{u}_t(t,\xi)|^2\leqslant 3E_0[\hat{u}](t,\xi)
\end{align*}
for any $t\geqslant0$, where we used $\rho(|\xi|)\mu(|\xi|)\leqslant 1$ and \eqref{Ineq_03} again. Summarizing the obtained estimates, we conclude the following chain:
\begin{align*}
	E_0[\hat{u}](t,\xi)&\leqslant \frac{1}{1-\beta}E[\hat{u}](t,\xi)\leqslant\frac{3}{1-\beta}\mathrm{e}^{-c\rho(|\xi|)t}E_0[\hat{u}](0,\xi).
\end{align*}
The proof has been finished.
\end{proof}
\subsection{Estimates for some Fourier multipliers}
To describe decay properties influenced by general strong damping for small frequencies, i.e. the Fourier multiplier $\chi_{\intt}(\xi)\mathrm{e}^{-c\rho(|\xi|)t}$, we may introduce the set $A_{\intt}^{n,s}$ containing some behaviors for $\mu(r)$ as $r\leqslant\varepsilon\ll 1$.
\begin{defn}\label{Defn_A_intt}
Let $s\in\mb{R}$ and $0<\varepsilon\ll 1$. We introduce a set
\begin{align*}
A_{\intt}^{n,s}:=\left\{\alpha_{n,s}\geqslant0:\int_0^{\varepsilon}r^{2s+n-1-2\alpha_{n,s}}\mu(r)^{-\alpha_{n,s}}\mathrm{d}r<\infty \right\}.
\end{align*}
Moreover, we denote by $\alpha_{n,s}^m:=\sup A_{\intt}^{n,s}<\infty$ which means $\alpha_{n,s}^m-2\delta\in A_{\intt}^{n,s}$ for any $\delta>0$.
\end{defn}
\begin{lemma}\label{Lem_Decay_Small}
Let us assume $\mu\in\ml{C}([0,\infty))$ with Hypothesis \ref{Hyp_Diss}. Then, the following estimates for the Fourier multiplier in small frequency zone hold:
\begin{align*}
\left\|\chi_{\intt}(\xi)|\xi|^s\mathrm{e}^{-c|\xi|^2\mu(|\xi|)t} \right\|_{L^2}\lesssim(1+t)^{-\frac{1}{2}\alpha_{n,s}^m+\delta},
\end{align*}
with $c>0$ and $s>-n/2$ for any $\delta>0$ and any $t>0$, where the number $\alpha_{n,s}^m$ was introduced in Definition \ref{Defn_A_intt}.
\end{lemma}
\begin{exam}\label{Examp_I}
Let us consider the fractional power case $\mu(|\xi|)=|\xi|^{\beta}$ with $\beta>-2$. According to Lemma \ref{Lem_Decay_Small}, the set is re-considered as
\begin{align*}
A_{\intt}^{n,s}=\left\{\alpha_{n,s}\geqslant0: \int_0^{\varepsilon}r^{2s+n-1-(2+\beta)\alpha_{n,s}}\mathrm{d}r<\infty \right\}=\left\{0\leqslant \alpha_{n,s}<\frac{2s+n}{2+\beta} \right\},
\end{align*}
which leads to $\alpha_{n,s}^m=(2s+n)/(2+\beta)$. At this time, concerning $s>-n/2$, we get a decay estimate
\begin{align}\label{Estimate_Examp_1}
\left\|\chi_{\intt}(\xi)|\xi|^s\mathrm{e}^{-c|\xi|^2\mu(|\xi|)t} \right\|_{L^2}\lesssim(1+t)^{-\frac{2s+n}{2(2+\beta)}+\delta},
\end{align}
with a sufficiently small constant $\delta>0$. Actually, providing that one avoids using Lemma \ref{Lem_Decay_Small}, we can get sharp decay estimates by explicit computations (see, for example, \cite[Lemma 2.1]{Charao-daLuz-Ikehata=2013}) as follows:
\begin{align}\label{Sharp_Example}
\left\|\chi_{\intt}(\xi)|\xi|^s\mathrm{e}^{-c|\xi|^2\mu(|\xi|)t} \right\|_{L^2}&\lesssim\left(\int_0^{\varepsilon}r^{2s+n-1}\mathrm{e}^{-2cr^{2+\beta}t}\mathrm{d}r \right)^{1/2}\notag\\
&\lesssim(1+t)^{-\frac{2s+n}{2(2+\beta)}},
\end{align}
with $s>-n/2$. In the comparison with the sharp estimate \eqref{Sharp_Example}, the derived  estimate (by Lemma \ref{Lem_Decay_Small}) is almost sharp in the sense of an arbitrarily small loss $\delta>0$. That is to say: considering the fractional type differential operator $\mu(|D|)$, an application of Lemma \ref{Lem_Decay_Small} only generates an arbitrarily small loss on the decay rate.
\end{exam}
\begin{proof}[Proof of Lemma \ref{Lem_Decay_Small}]
By applying the change of variable, our target is reduced by
\begin{align*}
	\left\|\chi_{\intt}(\xi)|\xi|^s\mathrm{e}^{-c|\xi|^2\mu(|\xi|)t} \right\|_{L^2}^2&=\int_{|\xi|\leqslant\varepsilon}|\xi|^{2s}\mathrm{e}^{-2c|\xi|^2\mu(|\xi|)t}\mathrm{d}\xi\\
	&\lesssim\int_0^{\varepsilon}r^{2s+n-1}\mathrm{e}^{-2cr^2\mu(r)t}\mathrm{d}r=:I_1(t).
\end{align*}
For small time $t\leqslant 1$, thanks to our assumption $2s+n>0$ it is trivial that
\begin{align*}
	I_1(t)\lesssim\int_0^{\varepsilon}r^{2s+n-1}\mathrm{d}r=\frac{1}{2s+n}\varepsilon^{2s+n}\lesssim 1.
\end{align*}
Let us now turn to the case $t>1$. Benefiting from  Definition \ref{Defn_A_intt}, we rewrite and estimate $I_1(t)$ as
\begin{align*}
	I_1(t)&=(2ct)^{-\alpha_{n,s}^m+2\delta}\int_0^{\varepsilon}\left(2cr^2\mu(r)t\right)^{\alpha_{n,s}^m-2\delta}\mathrm{e}^{-2cr^2\mu(r)t}r^{2s+n-1-2\alpha_{n,s}^m+4\delta}\mu(r)^{-\alpha_{n,s}^m+2\delta}\mathrm{d}r\\
	&\lesssim t^{-\alpha_{n,s}^m+2\delta}\int_0^{\varepsilon}r^{2s+n-1-2(\alpha_{n,s}^m-2\delta)}\mu(r)^{-\alpha_{n,s}^m+2\delta}\mathrm{d}r\lesssim t^{-\alpha_{n,s}^m+2\delta},
\end{align*}
because of the fact that $\alpha_{n,s}^{m}-2\delta\in A_{\intt}^{n,s}$ for any $\delta>0$. All in all, the proof is completed.
\end{proof}
\begin{remark}
In Definition \ref{Defn_A_intt}, if $\sup A_{\intt}^{n,s}=\max A_{\intt}^{n,s}$ holds, then we take the constant $\delta=0$ because of $\alpha_{n,s}^m\in A_{\intt}^{n,s}$. Namely, we avoid the small $\delta$-loss in this case.
\end{remark}
\begin{remark}
In Definition \ref{Defn_A_intt}, provided that $A_{\intt}^{n,s}$ is unbounded from the above, we still can get decay estimates for the time-dependent function $I_1(t)$. In such case, there exists a sequence $\{\ell_j\}_{j\in\mb{N}}\subset A_{\intt}^{n,s}$ such that $\ell_j\to\infty$ as $j\to\infty$. By the same way as the proof of Theorem \ref{Thm_Energy_Decay}, one finds
\begin{align*}
	I_1(t)\lesssim t^{-\ell_j}\int_0^{\varepsilon}r^{2s+n-1-2\ell_j}\mu(r)^{-\ell_j}\mathrm{d}r\lesssim t^{-\ell_j}
\end{align*}
for $t>1$. It leads polynomial type decay estimates for $I_1(t)$ with arbitrary-order $t^{-\ell_j}$ $(j=1,2,\dots)$. In fact, an instance is $\beta=-2$ in Example \ref{Examp_I}, which gives an exponential decay estimate.
\end{remark}

We now turn to the situation for large frequencies.
\begin{lemma}\label{Lem_Decay_Large}
Let us assume $\mu\in\ml{C}([0,\infty))$ with Hypothesis \ref{Hyp_Diss}. The following estimates hold:
\begin{align*}
	\left\|\chi_{\extt}(\xi)|\xi|^s\mathrm{e}^{-\frac{ct}{\mu(|\xi|)}}\hat{f}(\xi)\right\|_{L^2}\lesssim\begin{cases}
		(1+t)^{-\ell}\|f\|_{\dot{H}^s_{\mu,\ell}}&\mbox{if}\ \ \lim\limits_{r\to\infty}\mu(r)=\infty,\\
		\mathrm{e}^{-ct}\|f\|_{\dot{H}^s}&\mbox{if}\ \ \lim\limits_{r\to\infty}\mu(r)<\infty,
	\end{cases}
\end{align*}
with $c>0$, $s\in\mb{R}$ and $\ell\geqslant0$ for any $t>0$.
\end{lemma}
\begin{remark}
There is no contradiction between Hypothesis \ref{Hyp_Diss} and the assumption $\lim\nolimits_{r\to\infty}\mu(r)=\infty$ (or $\lim\nolimits_{r\to\infty}\mu(r)<\infty$).
\end{remark}
\begin{remark}\label{Rem_2.5}
In Lemma \ref{Lem_Decay_Large}, provided that $\lim\nolimits_{r\to\infty}\mu(r)=\infty$, then we need further $\mu(|D|)^{\ell}$ regularity of $f=f(x)$ to obtain some decay estimates. However, this phenomenon disappears if $\lim\nolimits_{r\to\infty}\mu(r)<\infty$. Namely, we has derived a critical condition (or threshold) concerning the pseudo-differential operator $\mu(|D|)$ for regularity-loss decay properties, which is described by the symbol of $\mu(|D|)$ as follows:
\begin{align}\label{Condition_Regularity_Loss}
	\lim\limits_{|\xi|\to\infty}\mu(|\xi|)=\infty.
\end{align} 
To be specific, if the symbol of $\mu(|D|)$ satisfies the condition \eqref{Condition_Regularity_Loss}, then our desired polynomially decay estimates can be reached by assuming higher-regularities for $f(x)$. Otherwise, some exponential decay estimates without asking higher-regularities for $f(x)$ can be derived.
\end{remark}
\begin{exam}
Let us take into consideration of $\mu(|\xi|)=|\xi|^{\beta}$ again. According to Lemma \ref{Lem_Decay_Large}, we may immediately obtain
\begin{align*}
\left\|\chi_{\extt}(\xi)|\xi|^s\mathrm{e}^{-\frac{ct}{\mu(|\xi|)}}\hat{f}(\xi)\right\|_{L^2}\lesssim\begin{cases}
	(1+t)^{-\frac{\bar{\ell}}{\beta}}\|f\|_{\dot{H}^{s+\bar{\ell}}}&\mbox{if}\ \ \beta>0,\\
	\mathrm{e}^{-ct}\|f\|_{\dot{H}^s}&\mbox{if}\ \ \beta\leqslant0,
\end{cases}
\end{align*}
with $\bar{\ell}=\beta\ell$, $\ell\geqslant0$ and $s\geqslant0$ for any $t>0$. These estimates exactly coincide with the regularity-loss decay property in \cite[Lemma 2.4]{Ikehata-Iyota=2018}, and the exponential decay property in \cite[Proof of Theorem 1.1]{Ikehata-Natsume=2012}. In other words, for the fractional type differential operator $\mu(|D|)$, an application of Lemma \ref{Lem_Decay_Large} is really sharp.
\end{exam}
\begin{proof}[Proof of Lemma \ref{Lem_Decay_Large}]
For the situation $\lim\nolimits_{r\to\infty}\mu(r)<\infty$, it immediately gives
\begin{align*}
\left\|\chi_{\extt}(\xi)|\xi|^s\mathrm{e}^{-\frac{ct}{\mu(|\xi|)}}\hat{f}(\xi)\right\|_{L^2}&\lesssim\left(\sup\limits_{|\xi|\geqslant N\gg 1}\mathrm{e}^{-\frac{ct}{\mu(|\xi|)}}\right)\|\chi_{\extt}(\xi)|\xi|^s\hat{f}(\xi)\|_{L^2}\\
&\lesssim \mathrm{e}^{-ct}\|f\|_{\dot{H}^s},
\end{align*}
where we applied the Plancherel theorem. For another case $\lim\nolimits_{r\to\infty}\mu(r)=\infty$, it needs to be treated carefully by
\begin{align*}
\left\|\chi_{\extt}(\xi)|\xi|^s\mathrm{e}^{-\frac{ct}{\mu(|\xi|)}}\hat{f}(\xi)\right\|_{L^2}&\lesssim\sup\limits_{|\xi|\geqslant N\gg 1}\left(\mu(|\xi|)^{-\ell}\mathrm{e}^{-\frac{ct}{\mu(|\xi|)}}\right)\|\chi_{\extt}(\xi)\mu(|\xi|)^{\ell}|\xi|^s\hat{f}(\xi)\|_{L^2}\\
&\lesssim t^{-\ell}\sup\limits_{|\xi|\geqslant N\gg 1}\left( \left(\tfrac{ct}{\mu(|\xi|)}\right)^{\ell}\mathrm{e}^{-\frac{ct}{\mu(|\xi|)}} \right)\|\chi_{\extt}(D)\mu(|D|)^{\ell}|D|^sf\|_{L^2}\\
&\lesssim (1+t)^{-\ell}\|\mu(|D|)^{\ell}f\|_{\dot{H}^s}.
\end{align*}
Our proof is completed.
\end{proof}
\subsection{Decay estimates of solutions}\label{Section_Decay_Estimates}
At the beginning, let us state the first result on decay estimates for energy terms $|D|u(t,\cdot)$ and $u_t(t,\cdot)$ in the $\dot{H}^s$ norm with $s\geqslant0$.
\begin{theorem}\label{Thm_Energy_Decay}
Let us assume that Hypothesis \ref{Hyp_Diss} holds. Then, the solution to the Cauchy problem \eqref{Eq_Linear_General_Damped_Waves} fulfills the following decay estimates:
\begin{align*}
\|\,|D|u(t,\cdot)\|_{\dot{H}^{s}}+\|u_t(t,\cdot)\|_{\dot{H}^s}\lesssim
\begin{cases}
	(1+t)^{-\frac{1}{2}\min\left\{\alpha_{n,s+1}^m-2\delta,2\ell_0 \right\}}\|u_0\|_{\dot{H}^{s+1}_{\mu,\ell_0}\cap L^1}&\\
	\quad+(1+t)^{-\frac{1}{2}\min\left\{\alpha_{n,s}^m-2\delta,2\ell_1 \right\}}\|u_1\|_{\dot{H}^{s}_{\mu,\ell_1}\cap L^1}&\mbox{if}\ \ 	\lim\limits_{r\to\infty}\mu(r)=\infty,\\
	(1+t)^{-\frac{1}{2}\alpha_{n,s+1}^m+\delta}\|u_0\|_{\dot{H}^{s+1}\cap L^1}&\\
	\quad+(1+t)^{-\frac{1}{2}\alpha_{n,s}^m+\delta}\|u_1\|_{\dot{H}^{s}\cap L^1}&\mbox{if}\ \ \lim\limits_{r\to\infty}\mu(r)<\infty, 
\end{cases}
\end{align*}
with $s\geqslant0$ and $\ell_0,\ell_1\geqslant0$ for any $\delta>0$, where the numbers $\alpha_{n,s+k}^m-2\delta\in A_{\intt}^{n,s+k}$ for $k=0,1$.
\end{theorem}
\begin{remark}Similarly to the discussion in Remark \ref{Rem_2.5}, in the view of Theorem \ref{Thm_Energy_Decay} we discovered a new threshold for regularity-loss decay properties. Precisely, this threshold to the damped wave model \eqref{Eq_Linear_General_Damped_Waves} is described by the symbol for $\mu(|D|)$ of the general damping term, i.e.
	\begin{align*}
	\lim\limits_{|\xi|\to\infty}\mu(|\xi|)\begin{cases}
		<\infty\ :&\mbox{decay properties without regularity-loss},\\
		=\infty\ :&\mbox{decay properties with } \ell_0,\ell_1\mbox{-order of regularity-loss}.
	\end{cases}
	\end{align*}
This is one of novelties of our paper.
\end{remark}
\begin{remark}
In the recent mathematical literature (see, for example, \cite{D'Abbicco-Ebert=2014,Ikehata-Iyota=2018}) the major concern is the initial value problem
\begin{align*}
u_{tt}-\Delta u+|D|^{\theta}u_t=0 \ \ \mbox{with}\ \ u(0,x)=u_0(x),\ u_t(0,x)=u_1(x),
\end{align*}
where  $\theta\geqslant0$. Then, the authors in the previous researches would like to find a threshold $\theta_{\mathrm{thres}}=2$ in the scale $\{|D|^{\theta}\}_{\theta\geqslant0}$ of the damping term between regularity-loss decay properties and exponential stabilities (for large frequencies). Nevertheless, the scale $\{|D|^{\theta}\}_{\theta\geqslant0}$ is too rough to verify the critical damping. Motivated by the innovative works \cite{Ebert-Girardi-Reissig=2020,Tao=2009}, in Theorem \ref{Thm_Energy_Decay} we described the threshold by the symbol of $\mu(|D|)$ in the damping term $-\mu(|D|)\Delta u_t$ or $\mu(|D|)|D|^2u_t$ such that $\lim\nolimits_{|\xi|\to\infty}\mu(|\xi|)=\infty$. It means that even we choose the damping term
\begin{align*}
	\underbrace{\log\big(1+\log\big(1+\cdots\log\big(\log}_{k\ times \ \log}(1+|D|)\big)\big)\big)|D|^2u_t,
\end{align*}
with any $k\in\mb{N}$, as the damping term in the damped wave model \eqref{Eq_Linear_General_Damped_Waves}, we still can observe regularity-loss decay properties of solutions.
\end{remark}
\begin{remark}
Concerning the case $\lim\nolimits_{r\to\infty}\mu(r)=\infty$, due to the phenomenon of regularity-loss, the optimal choices of the losses are given by $2\ell_0=\alpha_{n,s+1}^m-2\delta$ and $2\ell_1=\alpha_{n,s}^m-2\delta$ for any $\delta>0$.
\end{remark}
\begin{exam}
	Let us consider the fractional type operator $\mu(|D|)=|D|^{\sigma}$ with $\sigma\in(-1,\infty)$ fulfilling Hypothesis \ref{Hyp_Diss}, namely, the (very strong) structurally damped waves
	\begin{align*}
		u_{tt}-\Delta u+(-\Delta)^{1+\frac{\sigma}{2}}u_t=0 \ \ \mbox{with}\ \ u(0,x)=u_0(x),\ u_t(0,x)=u_1(x).
	\end{align*}
By using Theorem \ref{Thm_Energy_Decay}, because of $A_{\intt}^{n,s+k}=\left\{0\leqslant\alpha_{n,s+k}<\frac{2s+2k+n}{2+\sigma} \right\}$ for $k=0,1$, we arrive at
\begin{align*}
&\|\,|D|u(t,\cdot)\|_{\dot{H}^s}+\|u_t(t,\cdot)\|_{\dot{H}^s}\lesssim\begin{cases}
	(1+t)^{-\min\left\{\frac{2s+2+n}{2(2+\sigma)}-\delta,\frac{\ell_0}{\sigma} \right\}}\|u_0\|_{\dot{H}^{s+1+\ell_0}\cap L^1}&\\
	\quad	+(1+t)^{-\min\left\{\frac{2s+n}{2(2+\sigma)}-\delta,\frac{\ell_1}{\sigma} \right\}}\|u_1\|_{\dot{H}^{s+\ell_1}\cap L^1}&\mbox{if}\ \ \sigma\in(0,\infty),\\
	(1+t)^{-\frac{2s+2+n}{2(2+\sigma)}+\delta}\|u_0\|_{\dot{H}^{s+1}\cap L^1}\\
	\quad+	(1+t)^{-\frac{2s+n}{2(2+\sigma)}+\delta}\|u_1\|_{\dot{H}^{s}\cap L^1}&\mbox{if}\ \ \sigma\in(-1,0],
\end{cases}
\end{align*}
with sufficiently small constant $\delta>0$ for $s\geqslant0$ and $\ell_0,\ell_1\geqslant0$. Our previous estimates almost coincide (in the sense of arbitrarily small loss $\delta>0$) with those in \cite[Theorem 1.1]{Ikehata-Iyota=2018} if $\sigma\in(0,\infty)$, and \cite[Theorem 2.1]{Charao-daLuz-Ikehata=2013} if $\sigma\in(-1,0]$.
\end{exam}
\begin{proof}[Proof of Theorem \ref{Thm_Energy_Decay}]
Let us divide our discussion into three parts. For $\xi\in\ml{Z}_{\intt}(\varepsilon)$, because our assumption $\lim\nolimits_{r\downarrow0}r\mu(r)=0$, the key function behaviors
\begin{align*}
\rho(r)=\frac{r^2\mu(r)}{1+r^2\mu(r)^2}\approx r^2\mu(r)\ \ \mbox{as}\ \ r\leqslant\varepsilon \ll 1.
\end{align*}
Therefore, the Plancherel theorem associated with Lemma \ref{Lem_Decay_Small} and Proposition \ref{Prop_Pointwise} implies
\begin{align*}
&\|\chi_{\intt}(D)|D|^{s+1}u(t,\cdot)\|_{L^2}+\|\chi_{\intt}(D)|D|^su_t(t,\cdot)\|_{L^2}\\
&\qquad=\|\chi_{\intt}(\xi)|\xi|^{s+1}\hat{u}(t,\xi)\|_{L^2}+\|\chi_{\intt}(\xi)|\xi|^s\hat{u}_t(t,\xi)\|_{L^2}\\
&\qquad\lesssim\left\|\chi_{\intt}(\xi)|\xi|^s\mathrm{e}^{-c|\xi|^2\mu(|\xi|)t}\left(|\xi|\hat{u}_0(\xi)+\hat{u}_1(\xi)\right)\right\|_{L^2}\\
&\qquad\lesssim\left\|\chi_{\intt}(\xi)|\xi|^{s+1}\mathrm{e}^{-c|\xi|^2\mu(|\xi|)t}\right\|_{L^2}\|\hat{u}_0\|_{L^{\infty}}+\left\|\chi_{\intt}(\xi)|\xi|^{s}\mathrm{e}^{-c|\xi|^2\mu(|\xi|)t}\right\|_{L^2}\|\hat{u}_1\|_{L^{\infty}}\\
&\qquad\lesssim (1+t)^{-\frac{1}{2}\alpha_{n,s+1}^m+\delta}\|u_0\|_{L^1}+(1+t)^{-\frac{1}{2}\alpha_{n,s}^m+\delta}\|u_1\|_{L^1},
\end{align*}
where the Hausdorff-Young inequality was used in the last chain.

When $\xi\in\ml{Z}_{\extt}(N)$, thanks to $\lim\nolimits_{r\to\infty}r\mu(r)=\infty$ we may claim that
\begin{align*}
	\rho(r)=\frac{r^2\mu(r)}{1+r^2\mu(r)^2}\approx\frac{1}{\mu(r)}\ \ \mbox{as}\ \ r\geqslant N\gg1,
\end{align*}
which leads to
\begin{align*}
&\|\chi_{\extt}(D)|D|^{s+1}u(t,\cdot)\|_{L^2}+\|\chi_{\extt}(D)|D|^su_t(t,\cdot)\|_{L^2}\\
&\qquad\lesssim\left\|\chi_{\extt}(\xi)|\xi|^s\mathrm{e}^{-\frac{ct}{\mu(|\xi|)}}\left(|\xi|\hat{u}_0(\xi)+\hat{u}_1(\xi)\right)\right\|_{L^2}\\
&\qquad\lesssim\begin{cases}
(1+t)^{-\ell_0}\|\mu(|D|)^{\ell_0}u_0\|_{\dot{H}^{s+1}}+(1+t)^{-\ell_1}\|\mu(|D|)^{\ell_1}u_1\|_{\dot{H}^{s}}&\mbox{if}\ \ \lim\limits_{r\to\infty}\mu(r)=\infty,\\
\mathrm{e}^{-ct}\left(\|u_0\|_{\dot{H}^{s+1}}+\|u_1\|_{\dot{H}^s}\right)&\mbox{if}\ \ \lim\limits_{r\to\infty}\mu(r)<\infty.
\end{cases}
\end{align*}
In the above, we employed Lemma \ref{Lem_Decay_Large} by choosing $\ell=\ell_0$ and $\ell=\ell_1$, respectively.

To end this proof, we realize an exponential decay estimate for $\xi\in\ml{Z}_{\bdd}(\varepsilon,N)$ since $\rho(r)\approx c>0$ in the case $\varepsilon\leqslant r\leqslant N$. Combining all derived estimates in the last discussion, we complete our desired estimate.
\end{proof}

In the next result, we estimate the solution itself in the $L^2$ norm for higher-dimensions by employing similar idea to the one in Theorem \ref{Thm_Energy_Decay}.
\begin{coro}\label{Coro_Solution_high}
Let us assume that Hypothesis \ref{Hyp_Diss} holds. Then, the solution to the Cauchy problem \eqref{Eq_Linear_General_Damped_Waves} for $n\geqslant 3$ fulfills the following decay estimates:
\begin{align*}
	\|u(t,\cdot)\|_{L^2}\lesssim
	\begin{cases}
		(1+t)^{-\frac{1}{2}\min\left\{\alpha_{n,0}^m-2\delta,2\ell_0 \right\}}\|u_0\|_{\dot{H}^{0}_{\mu,\ell_0}\cap L^1}&\\
		\quad+(1+t)^{-\frac{1}{2}\min\left\{\alpha_{n,-1}^m-2\delta,2\ell_1 \right\}}\|u_1\|_{\dot{H}^{-1}_{\mu,\ell_1}\cap L^1}&\mbox{if}\ \ 	\lim\limits_{r\to\infty}\mu(r)=\infty,\\
		(1+t)^{-\frac{1}{2}\alpha_{n,0}^m+\delta}\|u_0\|_{L^2\cap L^1}&\\
		\quad+(1+t)^{-\frac{1}{2}\alpha_{n,-1}^m+\delta}\|u_1\|_{L^2\cap L^1}&\mbox{if}\ \ \lim\limits_{r\to\infty}\mu(r)<\infty, 
	\end{cases}
\end{align*}
with $\ell_0,\ell_1\geqslant0$ for any $\delta>0$, where the numbers $\alpha_{n,k-1}^m-2\delta\in A_{\intt}^{n,k-1}$ for $k=0,1$.
\end{coro}
\begin{remark}
In the viewpoint of the solution itself estimate, even when $\lim\nolimits_{r\to\infty}\mu(r)=\infty$ the phenomenon of regularity-loss can be dropped in some special situations. In a nutshell, let us assume $u_0\equiv0$. From the inequality
\begin{align*}
	\|\chi_{\extt}(D)u_1\|_{\dot{H}^{-1}_{\mu,\ell_1}}=\left\|\chi_{\extt}(\xi)\mu(|\xi|)^{\ell_1}|\xi|^{-1}\hat{u}_1(\xi)\right\|_{L^2}\lesssim\|u_1\|_{L^2}
\end{align*}
if the function $\mu(r)\lesssim r^{1/\ell_1}$ holds for any $r\geqslant N\gg 1$, we claim no regularity-loss for the solution itself estimate now, e.g. $\mu(|D|)=\log(1+|D|)$.
\end{remark}
\begin{proof}[Proof of Corollary \ref{Coro_Solution_high}]
Since the proof is strictly following the one for Theorem \ref{Thm_Energy_Decay}, we just sketch the different part. From Proposition \ref{Prop_Pointwise}, it leads to
\begin{align}\label{Coro_Est}
	|\hat{u}(t,\xi)|\lesssim\mathrm{e}^{-c\rho(|\xi|)t}\left(|\hat{u}_0(\xi)|+\frac{1}{|\xi|}|\hat{u}_1(\xi)|\right).
\end{align}
Obviously, we notice a singularity $1/|\xi|$ for $|\xi|\downarrow0$. For this reason, we need to restrict $n\geqslant3$ to avoid it when we use Lemma \ref{Lem_Decay_Small}.
\end{proof}

Ending this section, let us introduce another assumption that does not be included in Hypothesis \ref{Hyp_Diss}. The next one can be the supplement of Hypothesis \ref{Hyp_Diss}.
\begin{hyp}\label{Hyp_Diss_C}
	Let us suppose that $\mu=\mu(r)$ is a non-negative function such that $\mu\in\ml{C}([0,\infty))$. Moreover, it satisfies
	\begin{align*}
		\lim\limits_{r\downarrow0}r\mu(r)=0\ \ \mbox{as well as}\ \ \lim\limits_{r\to\infty}r\mu(r)\geqslant c\geqslant0.
	\end{align*}
\end{hyp}
\noindent This assumption covers another situation for $\mu(r)$. According to Hypothesis \ref{Hyp_Diss_C}, the key function \eqref{Key_function} has another behavior for large frequencies as follows:
\begin{align*}
\rho(|\xi|)\approx |\xi|\ \ \mbox{for}\ \ \xi\in\ml{Z}_{\extt}(N)
\end{align*}
with $N\gg1$. Thus, we may deduce
\begin{align*}
\left\|\chi_{\extt}(\xi)|\xi|^s\mathrm{e}^{-c|\xi|t}\left(|\xi|\hat{u}_0(\xi)+\hat{u}_1(\xi)\right)\right\|_{L^2}\lesssim\mathrm{e}^{-ct}\left(\|u_0\|_{\dot{H}^{s+1}}+\|u_1\|_{\dot{H}^{s}}\right),
\end{align*}
which leads to the same results as the case $\lim\nolimits_{r\to\infty}\mu(r)<\infty$ in Theorems \ref{Thm_Energy_Decay} and \ref{Thm_Solution_Itself}. Under Hypothesis \ref{Hyp_Diss_C}, we also can obtain analytic smoothing phenomenon if $\lim\nolimits_{r\to\infty}r\mu(r)\geqslant c>0$.  A typical instance is $\mu(|D|)=|D|^{-2}\log(1+|D|^{2\sigma})$ with $\sigma\in(1/2,\infty)$, in other words,
\begin{align*}
u_{tt}-\Delta u+\log(1+|D|^{2\sigma})u_t=0 \ \ \mbox{with}\ \ u(0,x)=u_0(x),\ u_t(0,x)=u_1(x).
\end{align*}
Since $\mu(r)=r^{-2}\log(1+r^{2\sigma})$ with $\sigma\in(1/2,\infty)$ accommodate to Hypothesis \ref{Hyp_Diss_C} carrying the set
\begin{align*}
A_{\intt}^{n,s}&=\left\{\alpha_{n,s}\geqslant0:\int_0^{\varepsilon}r^{2s+n-1}\left(\log(1+r^{2\sigma})\right)^{-\alpha_{n,s}}\mathrm{d}r<\infty  \right\}\\
&=\left\{0\leqslant \alpha_{n,s}<\frac{2s+n}{2\sigma} \right\}.
\end{align*}
Then, we can derive
\begin{align*}
\|\,|D|u(t,\cdot)\|_{\dot{H}^s}+\|u_t(t,\cdot)\|_{\dot{H}^s}\lesssim (1+t)^{-\frac{2s+n+2}{4\sigma}+\delta}\|u_0\|_{\dot{H}^{s+1}\cap L^1}+(1+t)^{-\frac{2s+n}{4\sigma}+\delta}\|u_1\|_{\dot{H}^s\cap L^1}
\end{align*}
with a sufficiently small constant $\delta>0$ for $s\geqslant0$. The last estimate corresponds to the one in \cite[Theorem 3.1]{Charao-DAbbicco-Ikehata}.

\section{Asymptotic profiles with additionally $L^1$ initial data}\label{Section_Asym_Prof}
Throughout this section, we will make use of asymptotic representations of solution to derive the sharp estimates and asymptotic profiles of the solution under Hypothesis \ref{Hyp_Diss}. Due to the explicit computations as well as the refined Fourier analysis, some properties of the solution itself will be improved.
\subsection{Asymptotic behaviors of solution in the Fourier space}
Recalling the $|\xi|$-dependent differential equation in \eqref{Eq_Fourier_Linear}, the corresponding characteristic equation is provide by
\begin{align*}
\lambda^2+\mu(|\xi|)|\xi|^2\lambda+|\xi|^2=0,
\end{align*}
whose roots $\lambda_{\pm}=\lambda_{\pm}(|\xi|)$ can be expressed by
\begin{align*}
\lambda_{\pm}(|\xi|)=-\frac{\mu(|\xi|)|\xi|^2}{2}\pm\frac{|\xi|}{2}\sqrt{\mu(|\xi|)^2|\xi|^2-4}.
\end{align*}
The pairwise distinct characteristic roots in the above allow us to represent the solution to \eqref{Eq_Fourier_Linear} in the next form:
\begin{align}\label{Representation}
\hat{u}(t,\xi)=\underbrace{\frac{\lambda_+(|\xi|)\mathrm{e}^{\lambda_-(|\xi|)t}-\lambda_-(|\xi|)\mathrm{e}^{\lambda_+(|\xi|)t}}{\lambda_+(|\xi|)-\lambda_-(|\xi|)}}_{=:\widehat{K}_0(t,|\xi|)}\hat{u}_0(\xi)+\underbrace{\frac{\mathrm{e}^{\lambda_+(|\xi|)t}-\mathrm{e}^{\lambda_-(|\xi|)t}}{\lambda_+(|\xi|)-\lambda_-(|\xi|)}}_{=:\widehat{K}_1(t,|\xi|)}\hat{u}_1(\xi).
\end{align}
In the next parts, we will employ WKB analysis to explore some asymptotic behaviors (or estimates) of these kernels in different local phase spaces.

\medskip
\noindent\underline{Estimates for bounded frequencies:} Our goal for $\xi\in\ml{Z}_{\bdd}(\varepsilon,N)$ is to deduce an exponential decay estimate for regular  data. Different from the previous studies, e.g. \cite[Section 2.3]{Jachmann-Reissig=2009} or \cite[Section 4]{Reissig=2016}, we can make use of the result from energy estimates rather than a contradiction argument associated with compactness of the frequency zones.  To be specific, we may observe from Proposition \ref{Prop_Pointwise} that
\begin{align*}
\chi_{\bdd}(\xi)|\hat{u}(t,\xi)|&\lesssim\chi_{\bdd}(\xi)\frac{1}{|\xi|}\mathrm{e}^{-c\rho(|\xi|)t}\left(|\xi|\,|\hat{u}_0(\xi)|+|\hat{u}_1(\xi)|\right)\\
&\lesssim\chi_{\bdd}(\xi)\mathrm{e}^{-ct}\left(|\hat{u}_0(\xi)|+|\hat{u}_1(\xi)|\right),
\end{align*}
with $c>0$. Hence, in remaining discussions, it is necessary to derive the asymptotic profiles for $\xi\in\ml{Z}_{\intt}(\varepsilon)\cup\ml{Z}_{\extt}(N)$ with $\varepsilon\ll 1$ as well as $N\gg1$, separately.

\medskip
\noindent\underline{Asymptotic behaviors for large frequencies:} Concerning $|\xi|\geqslant N\gg1$, the characteristic roots can be rewritten and expanded by
\begin{align*}
\lambda_{\pm}(|\xi|)&=-\frac{\mu(|\xi|)|\xi|^2}{2}\pm\frac{\mu(|\xi|)|\xi|^2}{2}\sqrt{1-\frac{4}{\mu(|\xi|)^2|\xi|^2}}\\
&=-\frac{\mu(|\xi|)|\xi|^2}{2}\pm\frac{\mu(|\xi|)|\xi|^2}{2}\left(1-\frac{2}{\mu(|\xi|)^2|\xi|^2}+\ml{O}\left(\mu(|\xi|)^{-4}|\xi|^{-4}\right)\right),
\end{align*}
where we used $\lim\nolimits_{|\xi|\to\infty}(\mu(|\xi|)|\xi|)^{-1}=0$ in Hypothesis \ref{Hyp_Diss}. That is to say
\begin{align*}
\lambda_+(|\xi|)=-\frac{1}{\mu(|\xi|)}+\ml{O}\left(\mu(|\xi|)^{-3}|\xi|^{-2}\right)\ \ \mbox{and}\ \ \lambda_-(|\xi|)=-\mu(|\xi|)|\xi|^2+\ml{O}\left(\mu(|\xi|)^{-1}\right).
\end{align*}
By plugging the last expansions into the representation \eqref{Representation}, the kernels in the Fourier space are
\begin{align*}
\widehat{K}_0(t,|\xi|)&=\frac{\left(-\frac{1}{\mu(|\xi|)}+\ml{O}\left(\mu(|\xi|)^{-3}|\xi|^{-2}\right) \right)\exp\big(\left(-\mu(|\xi|)|\xi|^2+\ml{O}(\mu(|\xi|)^{-1})\right)t\big)}{\mu(|\xi|)|\xi|^2+\ml{O}(\mu(|\xi|)^{-1})}\\
&\quad-\frac{\left(-\mu(|\xi|)|\xi|^2+\ml{O}(\mu(|\xi|)^{-1})\right)\exp\left(\left(-\frac{1}{\mu(|\xi|)}+\ml{O}\left(\mu(|\xi|)^{-3}|\xi|^{-2}\right) \right)t\right)}{\mu(|\xi|)|\xi|^2+\ml{O}(\mu(|\xi|)^{-1})}
\end{align*}
and
\begin{align*}
\widehat{K}_1(t,|\xi|)=\frac{\exp\left(\left(-\frac{1}{\mu(|\xi|)}+\ml{O}\left(\mu(|\xi|)^{-3}|\xi|^{-2}\right) \right)t\right)-\exp\big(\left(-\mu(|\xi|)|\xi|^2+\ml{O}(\mu(|\xi|)^{-1})\right)t\big)}{\mu(|\xi|)|\xi|^2+\ml{O}(\mu(|\xi|)^{-1})}
\end{align*}
for large frequencies $\xi\in\ml{Z}_{\extt}(N)$. So, we estimate
\begin{align*}
\chi_{\extt}(\xi)|\widehat{K}_0(t,|\xi|)|&\lesssim\chi_{\extt}(\xi)\left(\frac{\mathrm{e}^{-c\mu(|\xi|)|\xi|^2t}}{\mu(|\xi|)^2|\xi|^2}+\mathrm{e}^{-\frac{ct}{\mu(|\xi|)}}\right)\lesssim\chi_{\extt}(\xi)\mathrm{e}^{-\frac{ct}{\mu(|\xi|)}},\\
\chi_{\extt}(\xi)|\widehat{K}_1(t,|\xi|)|&\lesssim\frac{\chi_{\extt}(\xi)}{\mu(|\xi|)|\xi|^2}\left(\mathrm{e}^{-c\mu(|\xi|)|\xi|^2t}+  \mathrm{e}^{-\frac{ct}{\mu(|\xi|)}}\right)\lesssim\frac{\chi_{\extt}(\xi)}{\mu(|\xi|)|\xi|^2}\mathrm{e}^{-\frac{ct}{\mu(|\xi|)}},
\end{align*}
for any $t>0$, where we considered Hypothesis \ref{Hyp_Diss} again.

\medskip
\noindent\underline{Asymptotic behaviors for small frequencies:} Concerning $|\xi|\leqslant\varepsilon\ll 1$, the situation will be changed completely. Particularly, the characteristic roots will be expanded as follows:
\begin{align*}
\lambda_{\pm}(|\xi|)&=-\frac{\mu(|\xi|)|\xi|^2}{2}\pm i|\xi|\sqrt{1-\frac{\mu(|\xi|)^2|\xi|^2}{4}}\\
&=\pm i|\xi|-\frac{\mu(|\xi|)|\xi|^2}{2}+\ml{O}\left(\mu(|\xi|)^2|\xi|^3\right),
\end{align*}
where we considered $\lim\nolimits_{|\xi|\downarrow0}(\mu(|\xi|)|\xi|)=0$ in Hypothesis \ref{Hyp_Diss}. Furthermore, it has
\begin{align*}
\lambda_+(|\xi|)-\lambda_-(|\xi|)=2i|\xi|+\ml{O}\left(\mu(|\xi|)^2|\xi|^3\right).
\end{align*}
As a consequence, we may claim
\begin{align*}
\widehat{K}_0(t,|\xi|)&=\frac{\left(i|\xi|-\frac{\mu(|\xi|)|\xi|^2}{2}+\ml{O}\left(\mu(|\xi|)^2|\xi|^3\right)\right)\exp\left(\left(-i|\xi|-\frac{\mu(|\xi|)|\xi|^2}{2}+\ml{O}\left(\mu(|\xi|)^2|\xi|^3\right)\right)t\right)}{2i|\xi|+\ml{O}\left(\mu(|\xi|)^2|\xi|^3\right)}\\
&\quad-\frac{\left(-i|\xi|-\frac{\mu(|\xi|)|\xi|^2}{2}+\ml{O}\left(\mu(|\xi|)^2|\xi|^3\right)\right)\exp\left(\left(i|\xi|-\frac{\mu(|\xi|)|\xi|^2}{2}+\ml{O}\left(\mu(|\xi|)^2|\xi|^3\right)\right)t\right)}{2i|\xi|+\ml{O}\left(\mu(|\xi|)^2|\xi|^3\right)}
\end{align*}
and
\begin{align*}
\widehat{K}_1(t,|\xi|)=\frac{\exp\left(\left(i|\xi|-\frac{\mu(|\xi|)|\xi|^2}{2}+\ml{O}\left(\mu(|\xi|)^2|\xi|^3\right)\right)t\right)- \exp\left(\left(-i|\xi|-\frac{\mu(|\xi|)|\xi|^2}{2}+\ml{O}\left(\mu(|\xi|)^2|\xi|^3\right)\right)t\right)}{2i|\xi|+\ml{O}\left(\mu(|\xi|)^2|\xi|^3\right)}
\end{align*}
for small frequencies $\xi\in\ml{Z}_{\intt}(\varepsilon)$. We now may obtain
\begin{align*}
\chi_{\intt}(\xi)|\widehat{K}_0(t,|\xi|)|&\lesssim\chi_{\intt}(\xi)|\cos(|\xi|t)|\mathrm{e}^{-c\mu(|\xi|)|\xi|^2t},\\
\chi_{\intt}(\xi)|\widehat{K}_1(t,|\xi|)|&\lesssim\chi_{\intt}(\xi)\frac{|\sin(|\xi|t)|}{|\xi|}\mathrm{e}^{-c\mu(|\xi|)|\xi|^2t},
\end{align*}
for any $t>0$. Roughly speaking, there are some challenges in estimates for $\widehat{K}_1(t,|\xi|)$ coming from the unclear combined effect of the oscillating structure $\sin(|\xi|t)$, dissipative part $\mathrm{e}^{-c\mu(|\xi|)|\xi|^2t}$ and singularities $|\xi|^{-1}$ as $\xi\in\ml{Z}_{\intt}(\varepsilon)$.

Summarizing the derived estimates in the above and applying the representation of solution \eqref{Representation} in the Fourier space, we can conclude the next sharp estimates.
\begin{prop}\label{Prop_RE_Est}
Let us assume $\mu\in\ml{C}([0,\infty))$ fulfilling Hypothesis \ref{Hyp_Diss}. Then, the following pointwise estimates for the solution of the Cauchy problem \eqref{Eq_Fourier_Linear} hold:
\begin{align*}
	\chi_{\intt}(\xi)|\hat{u}(t,\xi)|&\lesssim\chi_{\intt}(\xi)\mathrm{e}^{-c\mu(|\xi|)|\xi|^2t} \left(|\cos(|\xi|t)|\,|\hat{u}_0(\xi)|+\frac{|\sin(|\xi|t)|}{|\xi|}|\hat{u}_1(\xi)|\right),\\
	\chi_{\bdd}(\xi)|\hat{u}(t,\xi)|&\lesssim\chi_{\bdd}(\xi) \mathrm{e}^{-ct}\left( |\hat{u}_0(\xi)|+|\hat{u}_1(\xi)| \right),\\
	\chi_{\extt}(\xi)|\hat{u}(t,\xi)|&\lesssim \chi_{\extt}(\xi)\mathrm{e}^{-\frac{ct}{\mu(|\xi|)}}\left(|\hat{u}_0(\xi)|+\frac{1}{\mu(|\xi|)|\xi|^2}|\hat{u}_1(\xi)|\right),
\end{align*}
with $c>0$ for any $\xi\in\mb{R}^n$ and $t>0$.
\end{prop}
\begin{remark}
The pointwise estimates in Proposition \ref{Prop_Pointwise} established by energy methods in the Fourier space seem to be sharp in the sense that asymptotic behaviors of the key function $\rho(|\xi|)$, i.e. the approximation \eqref{Behavior_of_Key_Function}, coincide with the exponential functions of each zone in Proposition \ref{Prop_RE_Est}. Nevertheless, we may observe a crucial difference $|\sin(|\xi|t)|$ occurring for $\xi\in\ml{Z}_{\intt}(\varepsilon)$. This function will be controlled by boundedness when we used energy methods, and appear when we applied explicit asymptotic analysis. Additionally, the $|\xi|$-dependent coefficient of $\hat{u}_1$ for large frequencies has been improved $1/(\mu(|\xi|)|\xi|)$-order.
\end{remark}
\subsection{Estimates of the solution itself in lower-dimensions}
In Corollary \ref{Coro_Solution_high}, we have estimated $u(t,\cdot)$ in the $L^2$ norm for $n\geqslant 3$ whose restriction originates from the strong singularity $1/|\xi|$ in $\widehat{K}_1(t,|\xi|)$ as $|\xi|\downarrow 0$. Comparing with the estimate \eqref{Coro_Est} from an application of energy methods, we got sharper estimates in Proposition \ref{Prop_RE_Est}. It is the key point for improvements.
\begin{theorem}\label{Thm_Solution_Itself}
Let us assume that Hypothesis \ref{Hyp_Diss} holds. Then, the solution to the Cauchy problem \eqref{Eq_Linear_General_Damped_Waves} for $n\geqslant 1$ fulfills the following estimates:
\begin{align*}
	\|u(t,\cdot)\|_{L^2}\lesssim
	\begin{cases}
		(1+t)^{-\frac{1}{2}\min\left\{\alpha_{n,0}^m-2\delta,2\ell_0 \right\}}\|u_0\|_{\dot{H}^{0}_{\mu,\ell_0}\cap L^1}&\\
		\quad+(1+t)^{-\frac{1}{2}\min\left\{\alpha_{n,-1}^m-2\delta,2(\ell_1+1) \right\}}\|u_1\|_{\dot{H}^{-2}_{\mu,\ell_1}\cap L^1}&\mbox{if}\ \ 	\lim\limits_{r\to\infty}\mu(r)=\infty,\\
		(1+t)^{-\frac{1}{2}\alpha_{n,0}^m+\delta}\|u_0\|_{L^2\cap L^1}&\\
		\quad+(1+t)^{-\frac{1}{2}\alpha_{n,-1}^m+\delta}\|u_1\|_{L^2\cap L^1}&\mbox{if}\ \ \lim\limits_{r\to\infty}\mu(r)<\infty, 
	\end{cases}
\end{align*}
with $\ell_0,\ell_1\geqslant0$ for any $\delta>0$, where the numbers $\alpha_{n,k-1}^m-2\delta\in A_{\intt}^{n,k-1}$ for $k=0,1$.
\end{theorem}
\begin{remark}
	The last theorem improves the one in Corollary \ref{Coro_Solution_high} as follows:
\begin{enumerate}[(1)]
	\item we actually did not restrict ourselves on the dimensions since the singularity $1/|\xi|$ for $\xi\in\ml{Z}_{\intt}(\varepsilon)$ with small time has been compensated by $|\sin(|\xi|t)|$;
	\item if $\lim\limits_{r\to\infty}\mu(r)=\infty$, we get $(1+t)^{-1}$ improvement on the time-dependent function, and $\dot{H}^{-1}$ on the regularity for the $u_1$ data since the accurate estimate in the Fourier space with the factor $1/(\mu(|\xi|)|\xi|)$ for $\xi\in\ml{Z}_{\extt}(N)$. 
\end{enumerate}
\end{remark}
\begin{proof}[Proof of Theorem \ref{Thm_Solution_Itself}]
Applying Proposition \ref{Prop_RE_Est}, we now may arrive at
\begin{align*}
\left\|\chi_{\intt}(\xi)\hat{u}(t,\xi)\right\|_{L^2}&\lesssim\left\|\chi_{\intt}(\xi)\cos(|\xi|t)\mathrm{e}^{-c\mu(|\xi|)|\xi|^2t}\right\|_{L^2}\|u_0\|_{L^1}+\left\|\chi_{\intt}(\xi)\frac{\sin(|\xi|t)}{|\xi|}\mathrm{e}^{-c\mu(|\xi|)|\xi|^2t}\right\|_{L^2}\|u_1\|_{L^1}\\
&=:I_2(t)\|u_0\|_{L^1}+I_3(t)\|u_1\|_{L^1},
\end{align*}
where H\"older's inequality as well as the Hausdorff-Young inequality were employed. Distinctly from Lemma \ref{Lem_Decay_Small} associated with $|\cos(|\xi|t)|\leqslant 1$, we notice that
\begin{align*}
I_2(t)\lesssim\left\|\chi_{\intt}(\xi)\mathrm{e}^{-c\mu(|\xi|)|\xi|^2t}\right\|_{L^2}\lesssim (1+t)^{-\frac{1}{2}\alpha_{n,0}^{m}+\delta}
\end{align*}
for any $n\geqslant 1$ and any $\delta>0$. Next, we compute estimates for $I_3(t)$ with caution. For $t\leqslant 1$, we straightway gain
\begin{align*}
I_3(t)=t\left(\int_{|\xi|\leqslant\varepsilon}\mathrm{e}^{-2c\mu(|\xi|)|\xi|^2t}\frac{|\sin(|\xi|t)|^2}{(|\xi|t)^2}\mathrm{d}\xi\right)^{1/2}\lesssim 1
\end{align*}
by employing $|\sin(|\xi|t)|\leqslant |\xi| t$ since $|\xi|t\leqslant \varepsilon t\ll 1$. While $t\geqslant 1$, we straightly imitate follow the proof of Lemma \ref{Lem_Decay_Small}, i.e. $I_1(t)$ with $s=-1$, and achieve
\begin{align*}
I_3(t)\lesssim \left(\int_{|\xi|\leqslant\varepsilon}|\xi|^{-2}\mathrm{e}^{-2c\mu(|\xi|)|\xi|^2t}\mathrm{d}\xi\right)^{1/2}\lesssim t^{-\frac{1}{2}\alpha_{n,-1}^m+\delta}
\end{align*}
for any $n\geqslant 1$. For this sake, we draw the conclusion
\begin{align*}
		\left\|\chi_{\intt}(D)u(t,\cdot)\right\|_{L^2}&=\left\|\chi_{\intt}(\xi)\hat{u}(t,\xi)\right\|_{L^2}\\
		&\lesssim  (1+t)^{-\frac{1}{2}\alpha_{n,0}^{m}+\delta}\|u_0\|_{L^1}+(1+t)^{-\frac{1}{2}\alpha_{n,-1}^{m}+\delta}\|u_1\|_{L^1}.
\end{align*}

Let us turn toward the case $\xi\in\ml{Z}_{\extt}(N)$ with $\lim\nolimits_{r\to\infty}\mu(r)=\infty$ since another case can be directly followed the proof of Theorem \ref{Thm_Energy_Decay}. By taking account into 
\begin{align*}
\chi_{\extt}(\xi)\frac{\mathrm{e}^{-\frac{ct}{\mu(|\xi|)}}}{\mu(|\xi|)|\xi|^2}
&\lesssim (1+t)^{-1}\frac{\chi_{\extt}(\xi)}{|\xi|^2}\mathrm{e}^{-\frac{ct}{\mu(|\xi|)}},
\end{align*}
one derives
\begin{align*}
\chi_{\extt}(\xi)|\hat{u}(t,\xi)|\lesssim\chi_{\extt}(\xi)\mathrm{e}^{-\frac{ct}{\mu(|\xi|)}}\left(|\hat{u}_0(\xi)|+(1+t)^{-1}|\xi|^{-2}|\hat{u}_1(\xi)|\right).
\end{align*}
Finally, the use of Lemma \ref{Lem_Decay_Large}, we obtain
\begin{align*}
\left\|\chi_{\extt}(D)u(t,\cdot)\right\|_{L^2}\lesssim(1+t)^{-\ell_0}\|u_0\|_{\dot{H}^0_{\mu,\ell_0}}+(1+t)^{-\ell_1-1}\|u_1\|_{\dot{H}^{-2}_{\mu,\ell_1}},
\end{align*}
providing that $\lim\nolimits_{r\to\infty}\mu(r)=\infty$. Because of the exponential stability for bounded frequencies, our proof is completed.
\end{proof}
\subsection{Asymptotic profiles of kernels in the Fourier space}\label{Sub-Sec_Kernel_Profiles}
Before stating some approximations for the kernels for $\xi\in\ml{Z}_{\intt}(\varepsilon)$ and $\xi\in\ml{Z}_{\extt}(N)$, respectively, let us introduce some auxiliary functions
\begin{align*}
	\widehat{\ml{G}}_0(t,|\xi|):=\mathrm{e}^{-\frac{t}{\mu(|\xi|)}} \ \ &\mbox{and}\ \ \widehat{\ml{G}}_1(t,|\xi|):=\frac{1}{\mu(|\xi|)|\xi|^2}\mathrm{e}^{-\frac{t}{\mu(|\xi|)}},\\
	\widehat{\ml{H}}_0(t,|\xi|):=\cos(|\xi|t)\mathrm{e}^{-\frac{1}{2}\mu(|\xi|)|\xi|^2t}\ \ &\mbox{and}\ \ \widehat{\ml{H}}_1(t,|\xi|):=\frac{\sin(|\xi|t)}{|\xi|}\mathrm{e}^{-\frac{1}{2}\mu(|\xi|)|\xi|^2t}. 
\end{align*}
The definitions of these Fourier multipliers are strongly motivated by asymptotic behaviors of the kernels in the representation of solution as $\xi\in\ml{Z}_{\intt}(\varepsilon)\cup\ml{Z}_{\extt}(N)$. Two propositions in this subsection serve for the proof of our vital estimates in Theorem \ref{Thm_Asymptotic_Profiles} later.

For large frequencies, we just need to study the situation $\lim\nolimits_{|\xi|\to\infty}\mu(|\xi|)=\infty$ due to the fact that an exponential decay holds (without any loss of regularity) for the solution provided $\lim\nolimits_{r\to\infty}\mu(r)<\infty$ in Theorem \ref{Thm_Solution_Itself}.
\begin{prop}\label{Prop_Refine_large}
Let us assume that Hypothesis \ref{Hyp_Diss} and $\lim\nolimits_{r\to\infty}\mu(r)=\infty$ hold. Then, the following pointwise estimates for the difference between two Fourier multipliers hold:
\begin{align}
\left|\chi_{\extt}(\xi)\left(\widehat{K}_0(t,|\xi|)-\widehat{\ml{G}}_0(t,|\xi|)\right)\right|&\lesssim\chi_{\extt}(\xi)\frac{\mathrm{e}^{-\frac{ct}{\mu(|\xi|)}}}{\mu(|\xi|)^2|\xi|^2},\label{Difference_Large_1}\\
\left|\chi_{\extt}(\xi)\left(\widehat{K}_1(t,|\xi|)-\widehat{\ml{G}}_1(t,|\xi|)\right)\right|&\lesssim\chi_{\extt}(\xi)\frac{\mathrm{e}^{-\frac{ct}{\mu(|\xi|)}}}{\mu(|\xi|)^3|\xi|^4},\label{Difference_Large_2}
\end{align}
with $c>0$ for any $\xi\in\mb{R}^n$ and $t>0$.
\end{prop} 
\begin{proof}[Proof of Proposition \ref{Prop_Refine_large}]
In the first place, we are going to demonstrate our desired estimate \eqref{Difference_Large_1}. On the basic of the asymptotic representation in the last subsection, one finds
\begin{align*}
&\left|\chi_{\extt}(\xi)\left(\widehat{K}_0(t,|\xi|)-\widehat{\ml{G}}_0(t,|\xi|)\right)\right|\\
&\qquad\lesssim\chi_{\extt}(\xi)\frac{\mathrm{e}^{-c\mu(|\xi|)|\xi|^2t}}{\mu(|\xi|)^2|\xi|^2}+\chi_{\extt}(\xi)\left|\frac{\mu(|\xi|)|\xi|^2+\ml{O}(\mu(|\xi|)^{-1})}{\mu(|\xi|)|\xi|^2+\ml{O}(\mu(|\xi|)^{-1})}\mathrm{e}^{-\frac{t}{\mu(|\xi|)}+\ml{O}(\mu(|\xi|)^{-3}|\xi|^{-2})t}-\mathrm{e}^{-\frac{t}{\mu(|\xi|)}}\right|\\
&\qquad\lesssim\chi_{\extt}(\xi)\frac{\mathrm{e}^{-c\mu(|\xi|)|\xi|^2t}}{\mu(|\xi|)^2|\xi|^2}+\chi_{\extt}(\xi)\left|\frac{\mu(|\xi|)|\xi|^2+\ml{O}(\mu(|\xi|)^{-1})}{\mu(|\xi|)|\xi|^2+\ml{O}(\mu(|\xi|)^{-1})}-1\right|\mathrm{e}^{-\frac{t}{\mu(|\xi|)}+\ml{O}(\mu(|\xi|)^{-3}|\xi|^{-2})t}\\
&\qquad\quad+\chi_{\extt}(\xi)\mathrm{e}^{-\frac{t}{\mu(|\xi|)}}\left|\mathrm{e}^{\ml{O}(\mu(|\xi|)^{-3}|\xi|^{-2})t}-1\right|.
\end{align*}
Then, we may see
\begin{align*}
	&\left|\chi_{\extt}(\xi)\left(\widehat{K}_0(t,|\xi|)-\widehat{\ml{G}}_0(t,|\xi|)\right)\right|\\
	&\qquad\lesssim\chi_{\extt}(\xi)\frac{\mathrm{e}^{-c\mu(|\xi|)|\xi|^2t}}{\mu(|\xi|)^2|\xi|^2}+\chi_{\extt}(\xi)\frac{\ml{O}(\mu(|\xi|)^{-1})}{\mu(|\xi|)|\xi|^2+\ml{O}(\mu(|\xi|)^{-1})}\mathrm{e}^{-\frac{t}{\mu(|\xi|)}+\ml{O}(\mu(|\xi|)^{-3}|\xi|^{-2})t}\\
	&\qquad\quad+\chi_{\extt}(\xi)\ml{O}\left(\mu(|\xi|)^{-3}|\xi|^{-2}\right)t\mathrm{e}^{-\frac{t}{\mu(|\xi|)}}\left|\int_0^1\mathrm{e}^{\ml{O}(\mu(|\xi|)^{-3}|\xi|^{-2})ts}\mathrm{d}s\right|\\
	&\qquad\lesssim\chi_{\extt}(\xi)\left(\frac{\mathrm{e}^{-c\mu(|\xi|)|\xi|^2t}}{\mu(|\xi|)^2|\xi|^2}+\frac{\mathrm{e}^{-\frac{ct}{\mu(|\xi|)}}}{\mu(|\xi|)^2|\xi|^2}+\frac{t\mathrm{e}^{-\frac{ct}{\mu(|\xi|)}}}{\mu(|\xi|)^3|\xi|^2}\right)\lesssim\chi_{\extt}(\xi)\frac{\mathrm{e}^{-\frac{ct}{\mu(|\xi|)}}}{\mu(|\xi|)^2|\xi|^2},
\end{align*}
in which we used
\begin{align*}
	\chi_{\extt}(\xi)\frac{t}{\mu(|\xi|)}\mathrm{e}^{-\frac{t}{\mu(|\xi|)}}\lesssim\chi_{\extt}(\xi)\mathrm{e}^{-\frac{ct}{\mu(|\xi|)}} \ \ \mbox{with}\ \ c>0.
\end{align*}

For another, in order to prove the estimate \eqref{Difference_Large_2}, we may follow the same philosophy as the previous one and get
\begin{align*}
	&\left|\chi_{\extt}(\xi)\left(\widehat{K}_1(t,|\xi|)-\widehat{\ml{G}}_1(t,|\xi|) \right)\right|\\
	&\qquad\lesssim\chi_{\extt}(\xi)\frac{\mathrm{e}^{-c\mu(|\xi|)|\xi|^2t}}{\mu(|\xi|)|\xi|^2}+\chi_{\extt}(\xi)\left|\frac{1}{\mu(|\xi|)|\xi|^2+\ml{O}(\mu(|\xi|)^{-1})}-\frac{1}{\mu(|\xi|)|\xi|^2}\right|\mathrm{e}^{-\frac{t}{\mu(|\xi|)}+\ml{O}(\mu(|\xi|)^{-3}|\xi|^{-2})t}\\
	&\qquad\quad+\chi_{\extt}(\xi)\frac{\mathrm{e}^{-\frac{t}{\mu(|\xi|)}}}{\mu(|\xi|)|\xi|^2}\left|\mathrm{e}^{\ml{O}(\mu(|\xi|)^{-3}|\xi|^{-2})t}-1\right|\\
	&\qquad\lesssim\chi_{\extt}(\xi)\left(\frac{\mathrm{e}^{-c\mu(|\xi|)|\xi|^2t}}{\mu(|\xi|)|\xi|^2}+\frac{\mathrm{e}^{-\frac{ct}{\mu(|\xi|)}}}{\mu(|\xi|)^3|\xi|^4}\right).
\end{align*}
Here, one realizes that
\begin{align*}
\lim\limits_{r\to\infty}\left(\frac{\mathrm{e}^{-c\mu(r)r^2t}}{\mu(r)r^2}\cdot\frac{\mu(r)^3r^4}{\mathrm{e}^{-\frac{ct}{\mu(r)}}}\right)=\lim\limits_{r\to\infty}\left(\mu(r)^2r^2\mathrm{e}^{\frac{ct}{\mu(r)}\left(1-\mu(r)^2r^2\right)}\right)=0,
\end{align*}
by taking Hypothesis \ref{Hyp_Diss} and the condition $\lim\nolimits_{r\to\infty}\mu(r)=\infty$. Hence, it implies
\begin{align*}
	\chi_{\extt}(\xi)\left(\frac{\mathrm{e}^{-c\mu(|\xi|)|\xi|^2t}}{\mu(|\xi|)|\xi|^2}+\frac{\mathrm{e}^{-\frac{ct}{\mu(|\xi|)}}}{\mu(|\xi|)^3|\xi|^4}\right)\lesssim\chi_{\extt}(\xi)\frac{\mathrm{e}^{-\frac{ct}{\mu(|\xi|)}}}{\mu(|\xi|)^3|\xi|^4},
\end{align*}
which completes the proof immediately.
\end{proof}

\begin{prop}\label{Prop_Refine_small}
Let us assume that Hypothesis \ref{Hyp_Diss} holds. Then, the following pointwise estimates for the difference between two Fourier multipliers hold:
\begin{align}
\left|\chi_{\intt}(\xi)\left(\widehat{K}_0(t,|\xi|)-\widehat{\ml{H}}_0(t,|\xi|)\right)\right|&\lesssim\chi_{\intt}(\xi)\mu(|\xi|)|\xi|\mathrm{e}^{-c\mu(|\xi|)|\xi|^2t},\label{Difference_Small_1}\\
\left|\chi_{\intt}(\xi)\left(\widehat{K}_1(t,|\xi|)-\widehat{\ml{H}}_1(t,|\xi|)\right)\right|&\lesssim\chi_{\intt}(\xi)\mu(|\xi|)\mathrm{e}^{-c\mu(|\xi|)|\xi|^2t},\label{Difference_Small_2}
\end{align}
with $c>0$ for any $\xi\in\mb{R}^n$ and $t>0$.
\end{prop}
\begin{proof}[Proof of Proposition \ref{Prop_Refine_small}]
With the purpose of deriving \eqref{Difference_Small_1}, the kernel may read as
\begin{align*}
\widehat{\ml{H}}_0(t,|\xi|)=\frac{\left(\mathrm{e}^{i|\xi|t}+\mathrm{e}^{-i|\xi|t}\right)i|\xi|}{2i|\xi|}\mathrm{e}^{-\frac{1}{2}\mu(|\xi|)|\xi|^2t},
\end{align*}
and then, one arrives at
\begin{align*}
&\left|\chi_{\intt}(\xi)\left(\widehat{K}_0(t,|\xi|)-\widehat{\ml{H}}_0(t,|\xi|)\right)\right|\\
&\quad \lesssim\sum\limits_{\pm} \chi_{\intt}(\xi)\left|\frac{i|\xi|\mp\frac{\mu(|\xi|)|\xi|^2}{2}+\ml{O}(\mu(|\xi|)^2|\xi|^3)}{2i|\xi|+\ml{O}(\mu(|\xi|)^2|\xi|^3)}\mathrm{e}^{\left(\mp i|\xi|-\frac{\mu(|\xi|)|\xi|^2}{2}+\ml{O}(\mu(|\xi|)^2|\xi|^3)\right)t}-\frac{i|\xi|}{2i|\xi|}\mathrm{e}^{\left(\mp i|\xi|-\frac{\mu(|\xi|)|\xi|^2}{2}\right)t} \right|\\
&\quad \lesssim\sum\limits_{\pm}\chi_{\intt}(\xi)\left|\frac{\mp i\mu(|\xi|)|\xi|^3+\ml{O}(\mu(|\xi|)^2|\xi|^4)}{2i|\xi|(2i|\xi|+\ml{O}(\mu(|\xi|)^{2}|\xi|^3))}\right|\mathrm{e}^{\left(\mp i|\xi|-\frac{\mu(|\xi|)|\xi|^2}{2}+\ml{O}(\mu(|\xi|)^2|\xi|^3)\right)t}\\
&\quad\quad+\sum\limits_{\pm}\chi_{\intt}(\xi)\mathrm{e}^{\left(\mp i|\xi|-\frac{\mu(|\xi|)|\xi|^2}{2}\right)t}\ml{O}(\mu(|\xi|)^2|\xi|^3)t\left|\int_0^1\mathrm{e}^{\ml{O}(\mu(|\xi|)^2|\xi|^3)ts}\mathrm{d}s\right|.
\end{align*}
It yields
\begin{align*}
	\left|\chi_{\intt}(\xi)\left(\widehat{K}_0(t,|\xi|)-\widehat{\ml{H}}_0(t,|\xi|)\right)\right|&\lesssim\chi_{\intt}(\xi)\left(\mu(|\xi|)|\xi|+\mu(|\xi|)^2|\xi|^3t\right)\mathrm{e}^{-c\mu(|\xi|)|\xi|^2t}\\
	&\lesssim\chi_{\intt}(\xi)\mu(|\xi|)|\xi|\mathrm{e}^{-c\mu(|\xi|)|\xi|^2t}.
\end{align*}

Let us turn to the other estimate \eqref{Difference_Small_2}, where we reformulate $\widehat{\ml{H}}_1(t,|\xi|)$ as follows:
\begin{align*}
	\widehat{\ml{H}}_1(t,|\xi|)=\frac{\mathrm{e}^{i|\xi|t}-\mathrm{e}^{-i|\xi|t}}{2i|\xi|}\mathrm{e}^{-\frac{\mu(|\xi|)|\xi|^2}{2}t}.
\end{align*}
Consequently, it leads to
\begin{align*}
&\left|\chi_{\intt}(\xi)\left(\widehat{K}_1(t,|\xi|)-\widehat{\ml{H}}_1(t,|\xi|)\right)\right|\\
&\qquad\lesssim\sum\limits_{\pm}\chi_{\intt}(\xi)\left|\frac{1}{2i|\xi|+\ml{O}(\mu(|\xi|)^2|\xi|^3)}-\frac{1}{2i|\xi|}\right|\mathrm{e}^{\left(\pm i|\xi|-\frac{\mu(|\xi|)|\xi|^2}{2}+\ml{O}(\mu(|\xi|)^2|\xi|^3)\right)t}\\
&\qquad\quad+\sum\limits_{\pm}\chi_{\intt}(\xi)\frac{\mathrm{e}^{-\frac{\mu(|\xi|)|\xi|^2}{2}t}}{|\xi|}\left|\mathrm{e}^{\pm i|\xi|t}\left(\mathrm{e}^{\ml{O}(\mu(|\xi|)^2|\xi|^3)t}-1\right)\right|\\
&\qquad\lesssim\chi_{\intt}(\xi)\left(\mu(|\xi|)^2|\xi|+\mu(|\xi|)^2|\xi|^2t\right)\mathrm{e}^{-c\mu(|\xi|)|\xi|^2t}\lesssim\chi_{\intt}(\xi)\mu(|\xi|)\mathrm{e}^{-c\mu(|\xi|)|\xi|^2t},
\end{align*}
where Hypothesis \ref{Hyp_Diss} was used again in the last line. All in all, our proof is finished.
\end{proof}
\begin{remark}
In Proposition \ref{Prop_Refine_large}, under the condition $\lim\nolimits_{|\xi|\to\infty}\mu(|\xi|)=\infty$ by subtracting the Fourier multipliers $\widehat{\ml{G}}_j(t,|\xi|)$ for $j=0,1$, respectively, the estimates have been improved (comparing with the third estimate in Proposition \ref{Prop_RE_Est}) by the factor $1/(\mu(|\xi|)|\xi|)^2$ for $\xi\in\ml{Z}_{\extt}(N)$. Moreover, if one subtracts the Fourier multipliers $\widehat{\ml{H}}_j(t,|\xi|)$ for $j=0,1$, individually, then the corresponding estimates have been enhanced (comparing with the first estimate in Proposition \ref{Prop_RE_Est}) by the factor $\mu(|\xi|)|\xi|$ for $\xi\in\ml{Z}_{\intt}(\varepsilon)$. It suggests large-time approximations.
\end{remark}
\subsection{Asymptotic profiles of solution}
To begin with this part, let us take the operators
\begin{align*}
		{\ml{G}}_0(t,|D|):=\mathrm{e}^{-\frac{t}{\mu(|D|)}} \ \ &\mbox{and}\ \ {\ml{G}}_1(t,|D|):=\frac{1}{\mu(|D|)|D|^2}\mathrm{e}^{-\frac{t}{\mu(|D|)}},\\
	{\ml{H}}_0(t,|D|):=\cos(|D|t)\mathrm{e}^{-\frac{1}{2}\mu(|D|)|D|^2t}\ \ &\mbox{and}\ \ {\ml{H}}_1(t,|D|):=\frac{\sin(|D|t)}{|D|}\mathrm{e}^{-\frac{1}{2}\mu(|D|)|D|^2t},
\end{align*}
whose corresponding symbols with respect to spatial pseudo-differential parts were shown at the beginning of Subsection \ref{Sub-Sec_Kernel_Profiles}. From the action of these operators on the corresponding initial data $u_0$ or $u_1$, we may introduce two functions as follows:
\begin{align}\label{v_extt}
v_{\extt}(t,x)&:=\chi_{\extt}(D)\left(\ml{G}_0(t,|D|)u_0(x)+\ml{G}_1(t,|D|)u_1(x)\right),\\
v_{\intt}(t,x)&:=\chi_{\intt}(D)\left(\ml{H}_0(t,|D|)u_0(x)+\ml{H}_1(t,|D|)u_1(x)\right).\label{v_intt}
\end{align}

The first profile exists in the case $\lim\nolimits_{r\to\infty}\mu(r)=\infty$. It can be represented by 
\begin{align*}
\ml{G}_0(t,|D|)u_0(x)+\ml{G}_1(t,|D|)u_1(x)=\mathrm{e}^{-\frac{t}{\mu(|D|)}}\left(u_0(x)+\frac{1}{\mu(|D|)|D|^2}u_1(x)\right),
\end{align*}
which solves the evolution equation
\begin{align}\label{First_Prof_omega}
	\mu(|D|)\omega_t+\omega=0 \ \ \mbox{with}\ \ \omega(0,x)=u_0(x)+\frac{1}{\mu(|D|)|D|^2}u_1(x).
\end{align}
Therefore, the solution to the general strongly damped waves \eqref{Eq_General_Damped_Waves} can be partly described \eqref{First_Prof_omega} if $\lim\nolimits_{r\to\infty}\mu(r)=\infty$.

The second profile appears in all cases, and can be rewritten by
\begin{align*}
	\ml{H}_0(t,|D|)u_0(x)+\ml{H}_1(t,|D|)u_1(x)=\underbrace{\exp\left(-\frac{1}{2}\mu(|D|)|D|^2t\right)}_{\mbox{diffusion-like part}}\underbrace{\left(\cos(|D|t)u_0(x)+\frac{\sin(|D|t)}{|D|}u_1(x)\right)}_{\mbox{half-wave part}}.
\end{align*}
This means that the solution to the general strongly damped waves \eqref{Eq_General_Damped_Waves} can be partly described by the combination of
\begin{itemize}
	\item diffusion-like equation $\varphi_t-\frac{1}{2}\mu(|D|)\Delta\varphi=0$;
	\item half-wave equation $\psi_t\pm i|D|\psi=0$;
\end{itemize}
carrying suitable initial datum.

\begin{theorem}\label{Thm_Asymptotic_Profiles}
Let us assume that Hypothesis \ref{Hyp_Diss} as well as $\mu(r)\lesssim r^{-\frac{1}{2}+\delta_0}$ for any $r\leqslant\varepsilon\ll 1$ with sufficiently small $\delta_0>0$ hold. Then, the difference between the solution to the Cauchy problem \eqref{Eq_Linear_General_Damped_Waves} and \eqref{v_intt} (or \eqref{v_extt}$+$\eqref{v_intt}) for $n\geqslant 1$ fulfills the following refined estimates:
\begin{align*}
\|u(t,\cdot)-v_{\intt}(t,\cdot)-v_{\extt}(t,\cdot)\|_{L^2}&\lesssim (1+t)^{-\ell_0-2}\|u_0\|_{\dot{H}^{-2}_{\mu,\ell_0}}+(1+t)^{-\frac{1}{2}\alpha_{n,-1}^m+\delta-1}\|u_0\|_{L^1}\\
&\quad+(1+t)^{-\ell_1-3}\|u_1\|_{\dot{H}^{-4}_{\mu,\ell_1}}+(1+t)^{-\frac{1}{2}\alpha_{n,-2}^m+\delta-1}\|u_1\|_{L^1}
\end{align*}
if $\lim\nolimits_{r\to\infty}\mu(r)=\infty$; moreover,
\begin{align*}
	\|u(t,\cdot)-v_{\intt}(t,\cdot)\|_{L^2}&\lesssim (1+t)^{-\frac{1}{2}\alpha_{n,-1}^m+\delta-1}\|u_0\|_{L^2\cap L^1}+(1+t)^{-\frac{1}{2}\alpha_{n,-2}^m+\delta-1}\|u_1\|_{L^2\cap L^1}
\end{align*}
if $\lim\nolimits_{r\to\infty}\mu(r)<\infty$; with $\ell_0,\ell_1\geqslant0$ for any $\delta>0$, where the number $\alpha_{n,k-1}^m-2\delta\in A_{\intt}^{n,k-1}$ for $k=-1,0$.
\end{theorem}
\begin{remark}
The additional technical assumption $\mu(r)\lesssim r^{-\frac{1}{2}+\delta_0}$ for any $r\leqslant\varepsilon\ll 1$ with sufficiently small $\delta_0>0$ can be removed if we just study the case $n\geqslant 3$. For another, this consideration is always holding in our cases, e.g. $\mu(r)=r^{\epsilon}$ and $\mu(r)=\log(\mathrm{e}+r)$.
\end{remark}
\begin{remark}
We actually can rewrite the estimates in Theorem \ref{Thm_Solution_Itself} by
\begin{align*}
	\|u(t,\cdot)\|_{L^2}&\lesssim (1+t)^{-\ell_0}\|u_0\|_{\dot{H}^{0}_{\mu,\ell_0}}+(1+t)^{-\frac{1}{2}\alpha_{n,0}^m+\delta}\|u_0\|_{L^1}\\
	&\quad+(1+t)^{-\ell_1-1}\|u_1\|_{\dot{H}^{-2}_{\mu,\ell_1}}+(1+t)^{-\frac{1}{2}\alpha_{n,-1}^m+\delta}\|u_1\|_{L^1}
\end{align*}
if $\lim\nolimits_{r\to\infty}\mu(r)=\infty$; moreover,
\begin{align*}
	\|u(t,\cdot)\|_{L^2}&\lesssim (1+t)^{-\frac{1}{2}\alpha_{n,0}^m+\delta}\|u_0\|_{L^2\cap L^1}+(1+t)^{-\frac{1}{2}\alpha_{n,-1}^m+\delta}\|u_1\|_{L^2\cap L^1}
\end{align*}
if $\lim\nolimits_{r\to\infty}\mu(r)<\infty$. Therefore, by subtracting the function $v_{\intt}(t,\cdot)$ in the $L^2$ norm, there are enhanced decay rates $(1+t)^{-1+\frac{1}{2}(\alpha_{n,k-1}^m-\alpha_{n,k-2}^m)}$ to $\|u_k\|_{L^1}$ for $k=0,1$ carrying $\alpha_{n,k-1}^m-\alpha_{n,k-2}^m\approx 1$. Furthermore, if one subtracts $v_{\extt}(t,\cdot)$ in the $L^2$ norm, we not only get enhanced decay rate $(1+t)^{-2}$ but also obtain reduced regularity $\dot{H}^{-2}$ for Sobolev regular datum. These nomenclatures have been clarified in the introduction.
\end{remark}
\begin{remark}
According to the last remark, we may assert that $v_{\intt}(t,x)$ and $v_{\extt}(t,x)$ are the asymptotic profiles of solutions for the wave equation with general strong damping \eqref{Eq_Linear_General_Damped_Waves} if Hypothesis \ref{Hyp_Diss} holds. Here, the condition \eqref{Condition_Regularity_Loss} is also a significant threshold for distinguishing different profiles of solution.
\end{remark}
\begin{proof}[Proof of Theorem \ref{Thm_Asymptotic_Profiles}]
We only show the verification of the first estimate in Theorem \ref{Thm_Asymptotic_Profiles} because the second one can be proved easily by the same procedure without asking for large frequencies (exponential decay estimates). By applying the Plancherel theorem, we deduce
\begin{align*}
&\|u(t,\cdot)-v_{\intt}(t,\cdot)-v_{\extt}(t,\cdot)\|_{L^2}\\
&\qquad\lesssim\sum\limits_{j=0,1}\left\|\chi_{\intt}(\xi)\left(\widehat{K}_j(t,|\xi|)-\widehat{\ml{H}}_j(t,|\xi|)\right)\right\|_{L^2}\|u_j\|_{L^1}+\mathrm{e}^{-ct}\left(\|u_0\|_{\dot{H}^{-2}_{\mu,\ell_0}}+\|u_1\|_{\dot{H}^{-4}_{\mu,\ell_1}}\right)\\
&\qquad\quad+\sum\limits_{j=0,1}\left\|\chi_{\extt}(\xi)\left(\widehat{K}_j(t,|\xi|)-\widehat{\ml{G}}_j(t,|\xi|)\right)\hat{u}_j(\xi)\right\|_{L^2}\\
&\qquad\lesssim \sum\limits_{j=0,1}\left\|\chi_{\intt}(\xi)\mu(|\xi|)|\xi|^{1-j}\mathrm{e}^{-c\mu(|\xi|)|\xi|^2t}\right\|_{L^2}\|u_j\|_{L^1}+\mathrm{e}^{-ct}\left(\|u_0\|_{\dot{H}^{-2}_{\mu,\ell_0}}+\|u_1\|_{\dot{H}^{-4}_{\mu,\ell_1}}\right)\\
&\qquad\quad+\sum\limits_{j=0,1}\left\|\chi_{\extt}(\xi)\mu(|\xi|)^{-2-j}|\xi|^{-2-2j}\mathrm{e}^{-\frac{ct}{\mu(|\xi|)}}\hat{u}_j(\xi)\right\|_{L^2},
\end{align*}
where we combined Propositions \ref{Prop_Refine_large} as well as \ref{Prop_Refine_small}. Eventually, one may use Lemmas \ref{Lem_Decay_Large} and \ref{Lem_Decay_Small} again to show
\begin{align*}
\left\|\chi_{\intt}(\xi)\mu(|\xi|)|\xi|^{1-j}\mathrm{e}^{-c\mu(|\xi|)|\xi|^2t}\right\|_{L^2}&\lesssim t^{-t}\left\|\chi_{\intt}(\xi)|\xi|^{-j-1}\mathrm{e}^{-c\mu(|\xi|)|\xi|^2t}\right\|_{L^2}\\
&\lesssim t^{-1-\frac{1}{2}\alpha_{n,-j-1}^m+\delta}
\end{align*}
for large-time, and
\begin{align*}
\left\|\chi_{\intt}(\xi)\mu(|\xi|)|\xi|^{1-j}\mathrm{e}^{-c\mu(|\xi|)|\xi|^2t}\right\|_{L^2}\lesssim\|\chi_{\intt}(\xi)\mu(|\xi|)\|_{L^2}\lesssim 1	
\end{align*}
for bounded-time by using the integrable condition; moreover,
\begin{align*}
	\left\|\chi_{\extt}(\xi)\mu(|\xi|)^{-2-j}|\xi|^{-2-2j}\mathrm{e}^{-\frac{ct}{\mu(|\xi|)}}\hat{u}_j(\xi)\right\|_{L^2}&\lesssim(1+t)^{-2-j}\left\|\chi_{\extt}(\xi)|\xi|^{-2-2j}\mathrm{e}^{-\frac{ct}{\mu(|\xi|)}}\hat{u}_j(\xi)\right\|_{L^2}\\
	&\lesssim (1+t)^{-2-j-\ell_j}\|u_j\|_{\dot{H}^{-2-2j}_{\mu,\ell_j}}.
\end{align*}
So, our demonstration is completed.
\end{proof}
\begin{remark}
By taking $\mu(r)=r^{\sigma}$ with $\sigma\in[0,\infty)$, our result in Theorem \ref{Thm_Asymptotic_Profiles} can almost coincide (with sufficiently small $\delta$-loss on the decay rate) with \cite[Theorem 1.1]{Ikehata=2014} when $\sigma=0$, \cite[Theorem 1.3]{Ikehata-Iyota=2018} when $\sigma=1$, \cite[Theorem 4.7]{Fukushima-Ikehata-Michihisa=2021} when $\sigma>0$. Of course, since a small loss in the decay rate, the threshold of regularity combined with dimension will be slightly modified. For example, in \cite[Estimate in (i), Theorem 4.7]{Fukushima-Ikehata-Michihisa=2021}, the first estimate holds only when $0\leqslant\ell\leqslant n/4-9/2-2\delta$ with $n>18$. Sometimes, the sufficiently small gap comparing with concrete examples cannot easily avoid because we are dealing with a quite general situation.
\end{remark}

\section{Final remarks}
As a pathfinder for wave models with general strong damping, we already determined some upper bound estimates of solutions in the $L^2$ norm, and asymptotic profiles of solutions by using refined WKB analysis. There are some interesting topics to be the continuity of this project, e.g. optimal estimates with weighted $L^1$ data and higher-order expansions of solutions in the general framework.

 So far the question for lower bound estimates of solution in the $L^2$ norm (carrying some suitable  hypotheses on $\mu(|D|)$) to verify the sharpness of Theorem \ref{Thm_Solution_Itself} is still open, but we believe Theorem \ref{Thm_Asymptotic_Profiles} will play a crucial role in such proof. Concerning the analytic class (with respect to small $r$ and $1/r$) for $\mu(r)$, we conjecture the optimal growth/decay estimates can be proved by following the procedure as those in \cite{Ikehata=2014,Ikehata-Onodera=2017,Ikehata-Iyota=2018}. Nonetheless, to consider this problem in the general framework on $\mu(|D|)$ rather than the analytic class is a challenging but interesting work.
 
 Another attractive question is the critical exponent, i.e. the threshold condition for global (in time) existence of solution and blow-up of solution, for semilinear Cauchy problem, namely,
 \begin{align}
 \begin{cases}
 u_{tt}-\Delta u-\mu(|D|)\Delta u_t=|u|^p,&x\in\mb{R}^n,\ t>0,\\
 u(0,x)=u_0(x),\ u_t(0,x)=u_1(x),&x\in\mb{R}^n,
 \end{cases}
 \end{align}
with $p>1$. By using our derived estimates in Theorem \ref{Thm_Solution_Itself} and \ref{Thm_Energy_Decay}, we conjecture that one may demonstrate global (in time) existence of small data Sobolev solution under some conditions on the exponent $p$. The main tools are based on the fractional Gagliardo-Nirenberg inequality, the fractional chain rule, the fractional Leibniz rule (see, for example, \cite{Palmieri-Reissig=2018}) and Banach's fixed point theory. However, to judge the critical exponent $p=p_{\mathrm{crit}}$, one needs to derive sharp blow-up condition for $p$, which is a completely open problem even for $\mu(|D|)\equiv1$.

\section*{Acknowledgments}
 The second author was supported in part by Grant-in-Aid for scientific Research (C) 20K03682 of JSPS. The authors thank Michael Reissig (TU Bergakademie Freiberg) for the suggestions in the preparation of the paper.

\end{document}